\documentclass[10pt,a4paper]{amsart}

\numberwithin{equation}{section}
\allowdisplaybreaks % let equations break across pages (doesn't work with groups like aligned)

\usepackage{mathtools,amssymb,eucal,mathrsfs,setspace,graphicx} % mathtools needed for cases* and brackets; loads amsmath
\usepackage[noadjust]{cite} % stops spaces being inserted before citations
\usepackage[margin=22mm,foot=10mm]{geometry}
\usepackage{lipsum}
\usepackage{tikz-cd} 

\usepackage{array}
\newcolumntype{C}{>{$}c<{$}} %Defines math mode in tabular

\usepackage[largesc,theoremfont]{newtxtext}      % mathptmx is now considered obsolete...
\usepackage[libertine,cmbraces,varbb]{newtxmath} % but the replacement greek math looks crap (hence libertine)...
\begingroup
  \makeatletter
  \@for\theoremstyle:=definition,remark,plain\do{%
    \expandafter\g@addto@macro\csname th@\theoremstyle\endcsname{%
      \addtolength\thm@preskip{.5\baselineskip plus .2\baselineskip minus .2\baselineskip}
      \addtolength\thm@postskip{.5\baselineskip plus .2\baselineskip minus .2\baselineskip}
    }%
  }
\endgroup
\usepackage{enumitem}
\setitemize{leftmargin=*}   % Removes margin from itemized lists (enumitem package)
\setenumerate{leftmargin=*, % Removes margin for enumerate lists (enumitem package)
  label=\textup{(\alph*)},  % Makes labels upright lowercase letters
	ref=\textup{\alph*}}      % Makes refs upright lowercase letters without parentheses
\usepackage{tikz}
\usetikzlibrary{arrows}
\usetikzlibrary{arrows.meta}
\usetikzlibrary{shapes}
\usepackage{bbm}
\usepackage[colorlinks=true,citecolor=red,linkcolor=blue]{hyperref} % load _before_ cleveref!
\usepackage{marginnote}
\usepackage[capitalise,noabbrev]{cleveref} % needed for \cref and \Cref (almost all refs except \eqref)

\newcommand{\pd}{\partial}     % holomorphic partial d

\renewcommand{\ge}{\geqslant} % never use \geq or \geqslant, use \ge which is globally defined to be one or the other
\renewcommand{\cong}{\simeq} % I like this symbol better...

\newcommand{\cc}{\mathsf{c}}   % central charge
\newcommand{\dd}{\mathrm{d}}   % d in derivatives and integrals
\newcommand{\ee}{\mathsf{e}}   % ln e = 1
\newcommand{\kk}{\mathsf{k}}   % level
\newcommand{\wun}{\vvmathbb{1}}  % here used to indicate characteristic functions (requires newtxmath)

\DeclarePairedDelimiter{\brac}{\lparen}{\rparen}   % use \brac for (...) and \brac* to automatically scale the ( and )
\DeclarePairedDelimiter{\sqbrac}{\lbrack}{\rbrack} % use \sqbrac[\big] for \bigl(...\bigr) etc...
\DeclarePairedDelimiter{\set}{\lbrace}{\rbrace}
\newcommand{\no}[1]{\mathopen{:} #1 \mathclose{:}} % normal ordering (prevent := or =:)

\DeclarePairedDelimiterX{\comm}[2]{\lbrack}{\rbrack}{#1 , #2}  % commutators
\DeclarePairedDelimiterX{\acomm}[2]{\lbrace}{\rbrace}{#1 , #2} % anticommutators
\DeclarePairedDelimiterX{\inner}[2]{\langle}{\rangle}{#1 , #2} % scalar products
\DeclarePairedDelimiterX{\super}[2]{\lparen}{\rparen}{#1 \delimsize\vert \mathopen{} #2} % for super args (m|n)

\newcommand{\blank}{{-}}

\newcommand{\fld}[1]{\mathbb{#1}}    % for fields and related things
\newcommand{\alg}[1]{\mathfrak{#1}}  % for Lie algebras
\newcommand{\grp}[1]{\mathsf{#1}}    % for groups
\newcommand{\VOA}[1]{\mathsf{#1}}    % VOAs
\newcommand{\ZZ}{\fld{Z}}
\newcommand{\CC}{\fld{C}}

\newcommand{\SLG}[2]{\grp{#1}_{#2}}            % Lie groups like SU(2)
\newcommand{\SLA}[2]{\alg{#1}_{#2}}                      % Lie algebras like sl(2)
\newcommand{\sltwo}{\SLA{sl}{2}}
\newcommand{\slthree}{\SLA{sl}{3}}
\newcommand{\slfour}{\SLA{sl}{4}}
\newcommand{\slnpone}{\SLA{sl}{n+1}}
\newcommand{\cartan}{\mathfrak{h}}

\newcommand{\cox}{\mathsf{h}}                                        % Coxeter number
\newcommand{\dcox}{\coroot{\cox}}                                    % dual Coxeter number

\newcommand{\roots}{\Delta}                                          % root system
\newcommand{\proots}{\Delta_+}                                       % positive roots
\newcommand{\sroots}{\Pi}                                            % simple root system
\newcommand{\sroot}[1]{\alpha_{#1}}                                  % simple root
\newcommand{\tsroot}[1]{\tilde{\alpha}_{#1}}                         % tidled simple root
\newcommand{\hroot}{\theta}                                          % highest root
\newcommand{\prvec}[1]{e_{#1}}                                        % root vector
\newcommand{\cvec}[1]{h_{#1}}                                        % Cartan basis vector
\newcommand{\nrvec}[1]{f_{#1}}                                       % negative root vector
\newcommand{\fwt}[1]{\omega_{#1}}                                    % fundamental weights
\newcommand{\coroot}[1]{#1^{\vee}}                                   % coroot
\newcommand{\nilorbit}[1]{\mathbbm{O}_{#1}}                                                 % nilpotent orbit

\newcommand{\fmin}{f_\theta}                                        % minimal nilpotent element
\newcommand{\frec}{f_\boxplus}                                     % rectangular nilpotent element
\newcommand{\fsub}{f_\textup{sub}}                                  % subregular nilpotent element
\newcommand{\fprin}{f_\textup{prin}}                                % regular/principal nilpotent element
\newcommand{\fhook}[1]{f^{(#1)}}                                  % subregular nilpotent element

\newcommand{\Wsymb}{\VOA{W}}

\newcommand{\uaffvoa}[2]{\VOA{V}^{#1}\brac{#2}}                   % universal affine VOA #1 = level, #2 = g
\newcommand{\saffvoa}[2]{\VOA{L}_{#1}\brac{#2}}                   % simple affine VOA #1 = level, #2 = g
\newcommand{\bgvoa}{\VOA{B}}                                      % bosonic ghost system
\newcommand{\halflattice}{\Pi}                          		  % half-lattice VOA

\newcommand{\uQHR}[3]{\Wsymb^{#1}\brac{#2,#3}}                    % universal QHR VOA
\newcommand{\twist}{\textup{tw}}                         		  % symbol for twisted things
\newcommand{\twzhu}[1]{\mathsf{Zhu}^{\twist}\sqbrac[\big]{#1}} 	  % twisted Zhu algebra
\theoremstyle{plain}
\newtheorem{theorem}{Theorem}[section]
\newtheorem{corollary}[theorem]{Corollary}
\newtheorem{lemma}[theorem]{Lemma}
\newtheorem{proposition}[theorem]{Proposition}

\newtheorem{mthm}{Main Result}

\newtheorem{definition}[theorem]{Definition}

\Crefname{assumption}{Assumption}{Assumptions}

\DeclareRobustCommand{\SkipTocEntry}[5]{} % ams workaround for silly toc ideals

\begin{document}

\title[Inverse reduction for hook-type W-algebras]{Inverse reduction for hook-type W-algebras}

\author[Z~Fehily]{Zachary Fehily}
\address[Zachary Fehily]{
School of Mathematics and Statistics \\
University of Melbourne \\
Parkville, Australia, 3010.
}
\email{fehilyz@unimelb.edu.au}

\begin{abstract}
Originating in the work of A.M. Semikhatov and D. Adamovi\'c, inverse reductions are embeddings involving W-algebras corresponding to the same Lie algebra but different nilpotent orbits. Here, we show that an inverse reduction embedding between the affine $\mathfrak{sl}_{n+1}$ vertex operator algebra and the minimal $\mathfrak{sl}_{n+1}$ W-algebra exists. This generalises the realisations for $n=1,2$ in \cite{AdaRea17,ACG21}. A similar argument is then used to show that inverse reduction embeddings exists between all hook-type $\mathfrak{sl}_{n+1}$ W-algebras, which includes the principal/regular, subregular, minimal $\mathfrak{sl}_{n+1}$ W-algebras, and the affine $\mathfrak{sl}_{n+1}$ vertex operator algebra. This generalises the regular-to-subregular inverse reduction of \cite{Feh21c}, and similarly uses free-field realisations and their associated screening operators.

\end{abstract}

\maketitle

\onehalfspacing

\section{Introduction}

\subsection{Background}
Given a simple Lie superalgebra $\mathfrak{g}$, a complex number $\kk$ and a nilpotent element $f \in \mathfrak{g}$, there exists a cochain complex of vertex operator algebras whose zeroth cohomology has the structure of a vertex operator algebra $\uQHR{\kk}{\mathfrak{g}}{f}$ called a \emph{W-algebra} \cite{KacQua03, KacQua04}. This process is called \emph{quantum hamiltonian reduction} of the universal affine vertex algebra $\uaffvoa{\kk}{\mathfrak{g}}$. 

 It is suspected one can also perform a \emph{partial} quantum hamiltonian reduction between two W-algebras originating from the same $\uaffvoa{\kk}{\mathfrak{g}}$ as long as their corresponding nilpotent orbits are related by a certain partial ordering $<$ \cite{Collingwood17} and possibly other conditions. Such partial reductions are in the same spirit as the `secondary quantum hamiltonian reductions' introduced by Madsen and Ragoucy \cite{RagSecond97}. In fact, a partial reduction for finite W-algebras has been defined for certain cases \cite{MorQua15, GenStage22}.

An \emph{inverse} to quantum hamiltonian reduction was first introduced by Semikhatov \cite{SemInv94} for the case of $\mathfrak{g}=\sltwo$. It takes the form of an embedding $\uaffvoa{\kk}{\sltwo} \hookrightarrow \uQHR{\kk}{\sltwo}{f} \otimes \halflattice$ where $\halflattice$ is a certain lattice-like vertex algebra. It was later shown by Adamovi\'c that this inverse reduction can be deployed to understand some of the representation theory of $\uaffvoa{\kk}{\sltwo}$ and, at certain levels, $\saffvoa{\kk}{\sltwo}$ \cite{AdaRea17}. An analogous embedding for $\slthree$ was studied in \cite{ACG21}, and an inverse to a partial reduction for $\slthree$ was studied in \cite{AdaRea20,Feh21b, AdaWeight23}. These inverse reductions have proven exceptionally useful in constructing interesting modules, as well as computing data important for conformal field theoretic applications.

One important task then is to construct inverse reductions for more general pairs of W-algebras, and unpack the relationships between the representation theories of the W-algebras involved. That is, for pairs of nilpotent elements $f_1, f_2 \in \mathfrak{g}$ whose nilpotent orbits satisfy $\nilorbit{f_1}<\nilorbit{f_2}$, construct and study embeddings of the form
\begin{equation}
    \uQHR{\kk}{\mathfrak{g}}{f_1} \hookrightarrow \uQHR{\kk}{\mathfrak{g}}{f_2} \otimes \VOA{V}.
\end{equation}

where $\VOA{V}$ is a tensor product of lattice and free-field vertex algebras, as well as analogous embeddings for the simple quotients of the W-algebras involved.

The main obstacles to constructing such embeddings for general $\mathfrak{g}$ and $f=f_1, f_2$ is that the operator product expansions of $\uQHR{\kk}{\mathfrak{g}}{f}$, if they are known, can be difficult to work with directly, and one has to come up with a suitable $\VOA{V}$. An alternative to brute forcing inverse reductions was described in \cite{Feh21c} for the $\mathfrak{g}=\slnpone$ principal-to-subregular inverse reduction, where one compares Wakimoto realisations of W-algebras and their corresponding screening operators from \cite{GenScreen20} to obtain the embedding.

The principal and subregular $\slnpone$ W-algebras, in addition to $\uaffvoa{\kk}{\slnpone}$ and the minimal $\slnpone$ W-algebra, are all examples of \emph{hook-type $\slnpone$ W-algebras} $\uQHR{\kk}{\slnpone}{\fhook{m}}$ for some $m=1,\dots,n+1$ (see \cref{def:hook}). Such W-algebras have attracted recent attention due to their appearance in Gaiotto and Rap\v c\' ak's Y-algebras from GL-twisted $\mathcal{N}=4$ supersymmetric gauge theories \cite{Gai19}. Consequent work has uncovered fascinating relationships amongst the vertex operator algebras appearing in Y-algebras, for example the triality isomorphisms of \cite{CreTri22}.

From the perspective of inverse reduction, hook-type nilpotent orbits satisfy $\nilorbit{f^{(m)}} > \nilorbit{f^{(m')}}$ when $m < m'$. Moreover, the `largest' and `smallest' $\slnpone$ W-algebras, with respect to the partial ordering induced by that on nilpotent orbits, are both hook-type. There are also Wakimoto realisations of $\uQHR{\kk}{\slnpone}{\fhook{m}}$ with particularly nice screening operators, see \eqref{eq:WakScreenersHook}.

The purpose of this paper is therefore to prove the existence of inverse reduction embeddings amongst hook-type $\slnpone$ W-algebras using a similar free-field approach as \cite{Feh21c}. That is embeddings, for $m<m'$, 
\begin{equation}
    \uQHR{\kk}{\slnpone}{\fhook{m'}} \hookrightarrow \uQHR{\kk}{\slnpone}{\fhook{m}} \otimes \VOA{V},
\end{equation}
for some vertex operator algebra $\VOA{V}$. This is useful as a proof-of-concept for the inverse reduction approach as a whole as it provides access a large number of nontrivial embeddings, in particular ones involving the affine $\uaffvoa{\kk}{\slnpone}$, that can be used in representation theoretic contexts. Secondly in light of their relationship to Y-algebras and the logarithmic conformal field-theoretic uses of the principal-to-subregular inverse reductions \cite{CreMod13,Feh21b,Feh21c}, inverse reductions for hook-type $\slnpone$ W-algebras may also assist in applications to physics.

\subsection{Outline}
We start in \cref{subsec:QHR} by describing the construction of W-algebras by quantum hamiltonian reduction of affine vertex operator algebras following \cite{KacQua03,KacQua04}. In \cref{subsec:Ordering}, two additional concepts related to quantum hamiltonian reduction are described: a partial version that is conjectured to exist between W-algebras, and an inverse. The construction of the latter is the topic of this paper.

The W-algebras of interest here, hook-type $\slnpone$ W-algebras, are described in \cref{subsec:Hook}. The full operator product expansions of hook-type W-algebras $\uQHR{\kk}{\slnpone}{\fhook{m}}$ where $m=1, \dots, n+1$ are not known in general but several properties of a set of strong generators can be established. For example the central charge \eqref{eq:hookcc} and some operator product expansions \eqref{eq:JPOPE}.

In order to relate hook-type $\slnpone$ W-algebras for different $m$ using inverse reduction, structural information in the form of free-field realisations is described in \cref{sec:WakWalg}. These free-field realisations are related to the Wakimoto realisation of the universal affine vertex operator algebra $\uaffvoa{\kk}{\slnpone}$ which is the topic of \cref{subsec:WakAff}. A crucial fact is that the image of the Wakimoto realisation of $\uaffvoa{\kk}{\slnpone}$ can be described in terms of screening operators as per \eqref{eq:WakAffScreen} \cite{Fre05Opers}. The Wakimoto realisation of the hook-type $\uQHR{\kk}{\slnpone}{\fhook{m}}$ is detailed in \cref{subsec:WakWalg} and also admits a screening operator description \eqref{eq:WakHookScreen} \cite{GenScreen20}.

In \cref{sec:Inverse}, these free-field realisations are used to show the existence of inverse reduction embeddings between hook-type $\slnpone$ W-algebras for generic $\kk$ (in the sense of \cite[Def.~4.5]{Gen17}). The first such embedding is obtained in \cref{subsec:MintoAff} and relates the affine $\uaffvoa{\kk}{\slnpone}$ ($m=(n+1)$ hook-type) to the minimal $\slnpone$ W-algebra $\uQHR{\kk}{\slnpone}{\fmin}$ ($m=n$ hook-type):

\begin{mthm}[\cref{thm:mintoaff}]
Let $\kk$ be generic, $\Pi$ be the half-lattice vertex algebra and $\bgvoa$ the bosonic ghost vertex algebra. There exists an embedding 
\begin{equation}
    \uaffvoa{\kk}{\slnpone} \hookrightarrow \uQHR{\kk}{\slnpone}{\fmin} \otimes \Pi \otimes \bgvoa^{\otimes(n-1)}
\end{equation}
whose image is specified by
\begin{equation}
    \uaffvoa{\kk}{\slnpone} \cong \textup{\normalfont ker} \int  \ee^{A^1+A^2}(z) \ \dd z,
\end{equation}
where $\ee^{A^1}(z)$ and $\ee^{A^2}(z)$ are screening fields for $\uQHR{\kk}{\slnpone}{\fmin}$ and $\Pi \otimes \bgvoa^{\otimes(n-1)}$ respectively.
\end{mthm}

This result is generalised in \cref{subsec:HooktoHook}. Here, a similar argument is used to describe an embedding relating the hook-type $\slnpone$ W-algebras $\uQHR{\kk}{\slnpone}{\fhook{m}}$ and $\uQHR{\kk}{\slnpone}{\fhook{m-1}}$.

\begin{mthm}[\cref{thm:HookToHook}]
Let $\kk$ be generic. There exists an embedding 
\begin{equation} \label{mthm:hooktohook}
    \uQHR{\kk}{\slnpone}{\fhook{m}} \hookrightarrow \uQHR{\kk}{\slnpone}{\fhook{m-1}} \otimes \Pi \otimes \bgvoa^{\otimes(m-2)}
\end{equation}
whose image is specified by
\begin{equation}
    \uaffvoa{\kk}{\slnpone} \cong \textup{\normalfont ker} \int  \ee^{A^1_m+A^2_m}(z) \ \dd z,
\end{equation}
where $\ee^{A^1_m}(z)$ and $\ee^{A^2_m}(z)$ are screening fields for $\uQHR{\kk}{\slnpone}{\fhook{m-1}}$ and $\Pi \otimes \bgvoa^{\otimes(m-2)}$ respectively.
\end{mthm}

To upgrade generic $\kk$ to nongeneric $\kk$, in \cref{subsec:Nongeneric} we apply the arguments of the previous sections to the hook-type $\slnpone$ W-algebra $\uQHR{U}{\slnpone}{\fhook{m}}$ over the ring $U=\CC[k]$. This culminates in \cref{thm:HookToHookNogen} which is the existence of the embedding \eqref{mthm:hooktohook} for noncritical $\kk$. This generalises the inverse reductions of \cite{AdaRea17} ($n=1,m=2$), \cite{AdaRea20} ($n=2, m=2$), \cite{Feh21c} ($m=2$) and \cite{ACG21} ($n=2, m=3$). 

Importantly, all inverse reductions obtained in this paper are for universal W-algebras. The question of when \eqref{mthm:hooktohook} descends to an embedding of simple quotients for general $n$ and $m$, and the subsequent representation theoretic consequences, are to be addressed in future work.

The main body of this paper concludes with outlining some future directions and outstanding questions in \cref{sec:Outlook}.

An upshot of the free-field approach used here is that, subject to knowing explicit formulae for the Wakimoto realisations of the W-algebras involved, the inverse reduction embedding can be written down. \cref{app:ExpExp} shows how this works for the embedding $\uaffvoa{\kk}{\slfour} \hookrightarrow \uQHR{\kk}{\slfour}{\fmin} \otimes \Pi \otimes \bgvoa^{\otimes(2)}$ which is described explicitly in \eqref{eq:sl4MintoAff}. 

\addtocontents{toc}{\SkipTocEntry}
\subsection*{Acknowledgements}
The author would like to thank David Ridout, Naoki Genra, Christopher Raymond, Drazen Adamovi\'c and Andrew Linshaw for their helpful comments, advice and suggestions. This work is supported by the Australian Research Council Discovery Project DP210101502.

% \newpage

\section{The vertex operator algebra \texorpdfstring{$\uQHR{\kk}{\slnpone}{\fhook{m}}$}{Wk(sln+1,fm)}} \label{sec:Walg}

\subsection{Quantum hamiltonian reduction} \label{subsec:QHR}
Originating in the work of Zamolodchikov and others \cite{ZamInf85,KPZ88,PolGau90,BerCon91}, and later generalised by various groups \cite{FF90,dBT94,KacQua03,KacQua04}, quantum hamiltonian reduction is a homological procedure that produces new vertex operator algebras from affine ones. The vertex operator algebras that result from quantum hamiltonian reduction are known as \emph{W-algebras}. Here we follow the general construction of W-algebras described in \cite[Sec.~1]{KacQua04}.

Let $\mathfrak{g}$ be a simple finite-dimensional Lie algebra, $\kk \in \CC$ and $x,f \in \mathfrak{g}$ satisfying
\begin{enumerate}
    \item The adjoint action $\textup{ad }x (\blank) = [x, \blank]$ of $x$ on $\mathfrak{g}$ is diagonalisable with half-integer eigenvalues:
    \begin{equation}
        \mathfrak{g} = \bigoplus_{m \in \frac{1}{2}\ZZ} \mathfrak{g}_m,
    \end{equation}
    where $\mathfrak{g}_m = \set{g \in \mathfrak{g} \ | \ [x,g] = m g}$.
    \item $f \in \mathfrak{g}_{-1}$.
    \item The adjoint action $\textup{ad }f$ of $f$ restricted to $\mathfrak{g}_{1/2}$ defines a vector space isomorphism $\mathfrak{g}_{1/2} \cong \mathfrak{g}_{-1/2}$.
\end{enumerate}
Alternatively, one can start with a nilpotent element $f \in \mathfrak{g}$ and a semisimple $x \in \mathfrak{g}$ such that the grading of $\mathfrak{g}$ by the adjoint action of $x$ is \emph{good grading} for $f$. That is, $f \in \mathfrak{g}_{-1}$ and $\text{ad}_f:\mathfrak{g}_j \rightarrow \mathfrak{g}_{j-1}$ is injective for $j \geq \frac{1}{2}$ and surjective for $j \leq  \frac{1}{2}$. These last two conditions imply condition (c) by setting $j=\frac{1}{2}$. Good $\frac{1}{2}\ZZ$-gradings for simple Lie algebras were classified in \cite{ElasClass05}. For example, any semisimple $x \in \mathfrak{g}$ that can be completed to an $\sltwo$ triple $\set{x=\frac{h}{2},e,f}$ in $\mathfrak{g}$ defines a good grading for that $f$ by $\sltwo$ representation theory. 

Let $A_{\textup{ne}} = \mathfrak{g}_{1/2}$. A symplectic form $(\blank | \blank)_{\textup{ne}}$ on $A_{\textup{ne}}$ is given by,
\begin{equation}
    (\varphi| \psi)_{\textup{ne}} = \inner{f}{[\varphi,\psi]}, \quad \varphi,\psi \in A_{\textup{ne}},
\end{equation}
where $\inner{\blank}{\blank}$ is the normalisation of the Killing form $\kappa_{\mathfrak{g}}$ of $\mathfrak{g}$ given by
\begin{equation}
   \inner{a}{b}=\frac{1}{2\dcox}\kappa_{\mathfrak{g}}(a,b).
\end{equation}
To see that this symplectic form is nondegenerate, the Killing form is invariant so $(\varphi| \psi)_{\textup{ne}} = \inner{[f,\varphi]}{\psi}$. As $\textup{ad }f:\mathfrak{g}_{1/2} \rightarrow \mathfrak{g}_{-1/2}$ is a vector space isomorphism, $(\blank | \blank)_{\textup{ne}}$ defines a nondegenerate pairing between $\mathfrak{g}_{1/2}$ and $\mathfrak{g}_{-1/2}$. 

We can therefore construct the bosonic ghost vertex operator algebra $\bgvoa(A_{\textup{ne}})$ whose generating (even) fields are denoted by $\set{\delta(z)}_{\delta \in A_{\textup{ne}}}$ with operator product expansions, for $\delta, \delta' \in A_{\textup{ne}}$,
\begin{equation}
    \delta(z) \delta'(w) \sim \frac{(\delta | \delta' )_{\textup{ne}}\wun(w)}{z-w}.
\end{equation}

Similarly, let $A_+ = \oplus_{m> 0} \ \mathfrak{g}_m$, $A_- = (A_+)^*$ and $A_{\textup{ch}} = A_+ \oplus A_-$. Define a nondegenerate symmetric form $(\blank | \blank)_{\textup{ch}}$ on $A_{\textup{ch}}$ by
\begin{equation}
    (A_+ | A_+)_{\textup{ch}} = (A_- | A_-)_{\textup{ch}} = 0, \quad (\delta | \xi)_{\textup{ch}} =  (\xi | \delta)_{\textup{ch}} = \xi(\delta), \quad 
    \delta \in A_+, \xi \in A_-.
\end{equation}
We can therefore construct the fermionic ghost vertex operator superalgebra $\VOA{F}(A_{\textup{ch}})$ whose generating (odd) fields are denoted by$\set{\varphi(z)}_{\varphi \in A_{\textup{ch}}}$ with operator product expansions, for $\varphi, \varphi' \in A_{\textup{ch}}$,
\begin{equation}
    \varphi(z) \varphi'(w) \sim \frac{(\varphi | \varphi' )_{\textup{ch}}\wun(w)}{z-w}.
\end{equation}

Let $\VOA{G} = \VOA{F}(A_{\textup{ch}}) \otimes \bgvoa(A_{\textup{ne}})$ and $\VOA{C} = \uaffvoa{\kk}{\mathfrak{g}}\otimes \VOA{G}$ where $\uaffvoa{\kk}{\mathfrak{g}}$ is the level-$\kk$ universal affine vertex operator algebra for $\mathfrak{g}$. The charge decomposition of $\VOA{F}(A_{\textup{ch}})$ that assigns charge $\pm 1$ to fields in $A_{\pm}$ respectively gives rise to a charge decomposition
\begin{equation}
    \VOA{C} = \bigoplus_{m \in \ZZ} \VOA{C}_m
\end{equation}
by setting the charge of all fields in $\bgvoa(A_{\textup{ne}})$ and $\uaffvoa{\kk}{\mathfrak{g}}$ to be zero. To construct a differential on the graded vertex algebra $\VOA{C}$, fix a basis $\set{J^a}_{a \in S_j}$ for each $\mathfrak{g}_j$. Let $S = \sqcup_j S_j$ (so that $\set{J^a}_{a \in S}$ is a basis of $\mathfrak{g}$), $S_+ = \sqcup_{j>0} S_j$ (so that $\set{J^a}_{a \in S_+}$ is a basis of $A_+$) and $C^{a,b}_c \in \CC$ be the structure constants of $\mathfrak{g}$ defined by
\begin{equation}
    [J^a,J^b] = \sum_{c \in S} C^{a,b}_c J^c.
\end{equation}
With respect to this basis, strong generators for $\uaffvoa{\kk}{\mathfrak{g}}$ are given by fields $\set{J^a(z)}_{a \in S}$, and their operator product expansions are, for $a,b \in S$,
\begin{equation} \label{eq:affOPE}
    J^{a}(z) J^{b}(w) \sim \frac{\kk \inner{ J^{a}}{ J^{b}} \wun(w)}{(z-w)^2} + \frac{\sum_{c \in S} C^{a,b}_c J^c(w)}{z-w}. 
\end{equation}

Denote the corresponding basis of $A_{\textup{ne}}$ by $\set{\delta^a}_{a \in S_{1/2}}$ and that of $A_+$ by $\set{\varphi^a}_{a \in S_+}$. Finally let $\set{\psi^a}_{a \in S_+}$ be the basis of $A_-$ dual to $\set{\varphi^a}_{a \in S_+}$ with respect to $(\blank | \blank)_{\textup{ch}}$. That is,
\begin{equation}
    (\varphi^a| \psi^b)_{\textup{ch}} = \psi^b(\varphi^a) = \delta_{a,b}.
\end{equation}
In terms of these generating fields, the charge decomposition of $\VOA{C}$ is defined by giving the fields $\set{J^a(z)}_{a \in S}$ and $\set{\delta^a(z)}_{a \in S_{1/2}}$ charge zero, $\set{\varphi^a(z)}_{a \in S_+}$ charge 1 and $\set{\psi^a(z)}_{a \in S_+}$ charge -1.

Define an odd field $d(z) \in \VOA{C}$ of charge equal to -1 by
\begin{align} \label{eq:differential}
    d(z) =& \sum_{a \in S_+} J^a(z) \psi^a(z) -\frac{1}{2}\sum_{a,b,c \in S_+} C^{a,b}_c \no{\varphi^c(z) \psi^a(z) \psi^b(z)}\\
          &\hspace{80pt}+ \sum_{a \in S_+} \inner{f}{J^a}\psi^a(z) + \sum_{a \in S_{1/2}}\psi^a(z) \delta^a(z), \notag
\end{align}
where we have omitted tensor product symbols. By \cite[Thm.~2.1]{KacQua03}, $d(z)d(w) \sim 0$. A simple consequence of this is that the zero mode $d=d_0: \VOA{C}_m \rightarrow \VOA{C}_{m-1}$ is a differential, i.e. $d^2 = 0$. 

The homology of the chain complex $(\mathsf{C},d)$ graded by charge, denoted by 
\begin{equation}
    \uQHR{\kk}{\mathfrak{g}}{x,f} =  H_{x,f}\left(\uaffvoa{\kk}{\mathfrak{g}} \right),
\end{equation}
is called a \emph{quantum hamiltonian reduction} of $\uaffvoa{\kk}{\mathfrak{g}}$, or a \emph{W-algebra} for short. The vertex algebra structure on $\uQHR{\kk}{\mathfrak{g}}{x,f}$ is inherited from that of $\VOA{C}$. As shown by Kac and Wakimoto \cite[Thm.~4.1]{KacQua03}, the homology is concentrated on the zeroth degree component,
\begin{equation}
    \uQHR{\kk}{\mathfrak{g}}{x,f} =  H_{x,f}^0 \left(\uaffvoa{\kk}{\mathfrak{g}} \right).
\end{equation}

To make $\uQHR{\kk}{\mathfrak{g}}{x,f}$ a vertex operator algebra, we require an energy-momentum field $L(z) \in \uQHR{\kk}{\mathfrak{g}}{x,f}$. This can be achieved by defining a field $L(z) \in \VOA{C}$ such that $d(z)$ is a primary field with conformal dimension $1$. With this goal in mind, let $\kk \neq - \dcox$, $\set{\xi^a}_{a \in S_{1/2}}$ be the dual basis to $\set{\delta^a}_{a \in S_{1/2}}$ with respect to $(\blank | \blank)_{\textup{ne}}$ and $\set{m_a}_{a \in S} \subset \CC$ be the set defined by $[x,J^a] = m_a J^a$. Define the field $L(z)$ by
\begin{align} \label{eq:Walgemfield}
    L(z) =&\ T^{\text{Sug.}}(z) + \pd x(z)
    - \sum_{a \in S_+} m_a \no{\psi^a(z) \pd \varphi^a(z)} \\&+\sum_{a \in S_+} (1-m_a) \no{\pd \psi^a(z) \varphi^a(z)}
    + \frac{1}{2}\sum_{a \in S_{1/2}} \no{\pd \delta(z)\xi^a(z)}. \notag
\end{align}
Kac, Roan and Wakimoto showed that $L(z)$ is an energy-momentum field in $\VOA{C}$ with central charge $\cc$ given by \cite[Thm.~2.2(a)]{KacQua03}
\begin{equation}
    \cc = \frac{\kk \textup{dim }\mathfrak{g}}{\kk+\dcox} - 12 \kk \inner{x}{x} - \sum_{a \in S_+} (12 m_a^2 - 12m_a + 2) - \frac{1}{2}\textup{dim }\mathfrak{g}_{1/2}
\end{equation}
and showed that $d(z)$ was primary of conformal dimension $1$ \cite[Thm.~2.3]{KacQua03}. The non-zero image of $L(z)$ in $\uQHR{\kk}{\mathfrak{g}}{x,f}$, which we also denote by $L(z)$, gives $\uQHR{\kk}{\mathfrak{g}}{x,f}$ the structure of a vertex operator algebra.

The most important examples of elements $x,f \in \mathfrak{g}$ satisfying the required conditions for quantum hamiltonian reduction are those arising from $\sltwo$ triples. In such cases, the nilpotent element $f \in \mathfrak{g}$ determines $x \in \mathfrak{g}$ up to conjugation by the Jacobson-Morozov theorem. For $x,f \in \mathfrak{g}$ belonging to $\sltwo$ triples, we will therefore denote the corresponding W-algebra by $\uQHR{\kk}{\mathfrak{g}}{f} =  H_{f}\left(\uaffvoa{\kk}{\mathfrak{g}} \right)$.

Alternatively, any element $x \in \mathfrak{g}$ whose adjoint action grades $\mathfrak{g}$ with half-integer eigenvalues and is good for some $f$ actually determines $f \in \mathfrak{g}$ up the action of the adjoint group of $\mathfrak{g}$ \cite[Rem.~1.1]{KacQua04}. The orbits of nilpotent elements in $\mathfrak{g}$ under the action of the adjoint group of $\mathfrak{g}$ are known as the \emph{nilpotent orbits} of $\mathfrak{g}$. Therefore the W-algebra $\uQHR{\kk}{\mathfrak{g}}{f}$ is determined, up to isomorphism, by the nilpotent orbit of $\mathfrak{g}$ containing $f$.

\subsection{Partial and inverse reduction} \label{subsec:Ordering}
There is a well known partial ordering on the nilpotent orbits of a given $\mathfrak{g}$ known as the \emph{closure} or \emph{Chevalley ordering} \cite{Collingwood17}. The Chevalley ordering of nilpotent orbits in $\mathfrak{g}$ therefore induces a partial ordering on the isomorphism classes of W-algebras obtained from $\uaffvoa{\kk}{\mathfrak{g}}$. 

It is suspected that, in addition to the usual quantum hamiltonian reduction, one can also perform a `partial quantum hamiltonian reduction' between two W-algebras as long as they are related by this partial ordering. Strong supporting evidence for this can be found in the twisted Zhu algebra of W-algebras: the associative algebra $\twzhu{\uQHR{\kk}{\mathfrak{g}}{f}}$ is isomorphic to the finite W-algebra corresponding to the same $\mathfrak{g}$ and $f$ \cite{DeSFin06}. 

Partial reductions of finite W-algebras corresponding to $\mathfrak{g} = \slnpone$ were conjectured to exist by Morgan \cite{MorQua15} and shown to exist  in \cite{GenStage22} when the corresponding nilpotent orbits $\mathbb{O}_{f_i}$ satisfy $\mathbb{O}_{f_1} < \mathbb{O}_{f_2}$ in the Chevalley ordering in addition to a number of sufficient conditions. While early work on partial quantum hamiltonian reductions exists in physics literature \cite{RagSecond97}, the vertex algebraic content of such a construction is currently being investigated.

An `inverse' to quantum hamiltonian reduction was first introduced by Semikhatov \cite{SemInv94} for the case of $\mathfrak{g}=\sltwo$. It takes the form of an embedding 
\begin{equation} \label{eq:printosubsl2}
    \uaffvoa{\kk}{\sltwo} \hookrightarrow \uQHR{\kk}{\sltwo}{f} \otimes \halflattice,
\end{equation}
where $\halflattice$ is a lattice vertex algebra defined in \cite{BerRep02}. It was later shown by Adamovi\'c that this inverse reduction can be deployed to understand some of the representation theory of $\uaffvoa{\kk}{\sltwo}$ and, at certain levels, $\saffvoa{\kk}{\sltwo}$ \cite{AdaRea17}. Similarly, inverse reduction for the minimal $\slthree$ W-algebra is known \cite{ACG21} and given by an embedding
\begin{equation}
    \uaffvoa{\kk}{\slthree} \hookrightarrow \uQHR{\kk}{\slthree}{f_\theta} \otimes \halflattice \otimes \bgvoa,
\end{equation}
where $\bgvoa$ is the $\beta \gamma$ ghost vertex algebra. 

It is expected that inverse quantum hamiltonian reductions can be constructed for more general pairs of W-algebras, where what is being inverted is a partial quantum hamiltonian reduction. Similar to the affine case, such a construction will provide nontrivial relationships between the representation theories of the W-algebras involved. For example, in the $\slthree$ case, there is an inverse reduction embedding \cite{AdaRea20}
\begin{equation} \label{eq:printosubsl3}
    \uQHR{\kk}{\slthree}{f_\theta} \hookrightarrow \uQHR{\kk}{\slthree}{f_1+f_2} \otimes \halflattice.
\end{equation}

This inverse reduction was used in \cite{Feh21b} to determine modular transformations and conjecture Grothendiek fusion rules for Bershadsky-Polyakov minimal models in terms of such data for $\VOA{W}_3$ minimal models. The same inverse reduction also facilitates a classification of irreducible weight modules (with finite-dimensional weight spaces) for $\uQHR{\kk}{\slthree}{f_\theta}$ and, for certain $\kk$, its simple quotient \cite{AdaWeight23}.  

The inverse reductions \eqref{eq:printosubsl2} and \eqref{eq:printosubsl3} were generalised to $\slnpone$ in \cite{Feh21c}. The W-algebras involved therein are the subregular W-algebra $\uQHR{\kk}{\slnpone}{\fsub}$, also known as the Feigen-Semikhatov algebra $\mathcal{W}^{(2)}_{n+1}$ \cite{FS04}, and the principal W-algebra $\uQHR{\kk}{\slnpone}{\fprin}$. The inverse reduction is an embedding 
\begin{equation}
    \uQHR{\kk}{\slnpone}{\fsub} \hookrightarrow \uQHR{\kk}{\slnpone}{\fprin} \otimes \halflattice,
\end{equation}
and this embedding was used to study the representation theory of the subregular W-algebra. A particularly interesting feature of this inverse reduction is that it descends to an embedding of simple quotients when $\kk$ is nondegenerate admissible \cite[Thm.~5.1]{Feh21c}. At such $\kk$, the simple quotient of $\uQHR{\kk}{\slnpone}{\fprin}$ is rational \cite{AraRat15}.

Recalling that the (conjectured) necessary condition for the existence of a partial quantum hamiltonian reduction from $\uQHR{\kk}{\mathfrak{g}}{f_1}$ to $\uQHR{\kk}{\mathfrak{g}}{f_2}$ is that $\mathbb{O}_{f_1} < \mathbb{O}_{f_2}$ , we believe that there exists embeddings of the form
\begin{equation} \label{eq:invembGeneral}
    \uQHR{\kk}{\mathfrak{g}}{f_1} \hookrightarrow \uQHR{\kk}{\mathfrak{g}}{f_2} \otimes \VOA{V},
\end{equation}
where $\VOA{V}$ is a vertex operator algebra representing the degrees of freedom `lost' in performing the partial quantum hamiltonian reduction from $ \uQHR{\kk}{\mathfrak{g}}{f_1}$ to $ \uQHR{\kk}{\mathfrak{g}}{f_2}$. In all known examples, $\VOA{V}$ consists of free-field vertex algebras and (half-)lattice vertex algebras.

\subsection{Hook-type \texorpdfstring{$\slnpone$}{sln+1} W-algebras} \label{subsec:Hook}
We now define the W-algebras of interest in this paper. Let $\mathfrak{g}=\slnpone$. Nilpotent orbits of $\slnpone$ are indexed by partitions of $n+1$, so we will denote nilpotent orbits by $\nilorbit{\lambda}$ where $\lambda=(\lambda_1, \lambda_2, \dots)$ is a partition of $n+1$ . The partial ordering on nilpotent orbits of $\slnpone$ is given by the dominance ordering on the corresponding partition \cite{Gers61}. That is, $\nilorbit{\lambda} \leq \nilorbit{\lambda'}$ if and only if $\lambda \leq \lambda'$.

The structure of this partial ordering is complicated in general but there are generic features that are noteworthy. For example, the largest nilpotent orbit is always $\nilorbit{\textup{prin}} = \nilorbit{(n+1)}$ (the \emph{principal} nilpotent orbit), while the smallest is $\nilorbit{(1^{n+1})} = \set{\mathbf{0}}$. The \emph{subregular} nilpotent orbit $\nilorbit{\textup{sub}}  = \nilorbit{(n,1)}$ is always smaller than $\nilorbit{\textup{reg}}$ and larger than all other nilpotent orbits. Similarly, the \emph{minimal} nilpotent orbit $\nilorbit{\textup{min}} = \nilorbit{(2,1^{n-1})}$ is always larger than $\nilorbit{(1^{n+1})}$ but smaller than all other nilpotent orbits. 

\begin{figure}[ht]
\centering
	\begin{tikzpicture}[xscale=2.5,yscale=2]
		\draw[->] 
		(1.5,0.6)  node[above] {$\uQHR{\kk}{\SLA{sl}{6}}{0} \cong \uaffvoa{\kk}{\SLA{sl}{6}}$} --
		(1.5,0.2);
		
		\draw[->,dashed] 
		(1.5,-0.2)  node[above] {$\uQHR{\kk}{\SLA{sl}{6}}{\fmin}$} --
		(1.5,-0.6) node[below] {$\uQHR{\kk}{\SLA{sl}{6}}{f_{(2,2,1,1)}}$};
		
		\draw[->,dashed] (2.1,-0.8) to [bend left=20] (3,-1.5) node[below]{$\uQHR{\kk}{\SLA{sl}{6}}{f_{(2,2,2)}}$};
		\draw[->,dashed] (0.9,-0.8) to [bend right=20] (0,-1.5)
		node[below]{$\uQHR{\kk}{\SLA{sl}{6}}{f_{(3,1,1,1)}}$};
		
		\draw[->,dashed] (3,-1.9) to [bend left=20] (2.1,-2.5);
		\draw[->,dashed] (0,-1.9) to [bend right=20] (0.9,-2.5);
		
		\draw[->,dashed] (2.1,-2.6) to [bend left=20] (3,-3.2) node[below]{$\uQHR{\kk}{\SLA{sl}{6}}{f_{(3,3)}}$};
		\draw[->,dashed] (0.9,-2.6) to [bend right=20] (0,-3.2)
		node[below]{$\uQHR{\kk}{\SLA{sl}{6}}{f_{(4,1,1)}}$};
		
		\draw (1.5,-2.35) node[below]{$\uQHR{\kk}{\SLA{sl}{6}}{f_{(3,2,1)}}$};
		
		\draw[->,dashed] (3,-3.6) to [bend left=20] (2.1,-4.2);
		\draw[->,dashed] (0,-3.6) to [bend right=20] (0.9,-4.2);
		
		\draw (1.5,-4.05) node[below]{$\uQHR{\kk}{\SLA{sl}{6}}{f_{(4,2)}}$};
		
		\draw[->, dashed] 
		(1.5,-4.45) --
		(1.5,-4.85) node[below] {$\uQHR{\kk}{\SLA{sl}{6}}{\fsub}$};
		
	    \draw[->, dashed] 
		(1.5,-5.25) --
		(1.5,-5.65) node[below] {$\uQHR{\kk}{\SLA{sl}{6}}{\fprin}$};
	\end{tikzpicture}
\caption{The partial ordering of W-algebras for $\SLA{sl}{6}$. Here we choose an element $f_\lambda$ from each nilpotent orbit $\nilorbit{\lambda}$. The nilpotent orbits in $\SLA{sl}{6}$ associated to the W-algebras increase in size from top to bottom with respect to the partial ordering. The W-algebras appearing at the same height (e.g. $\uQHR{\kk}{\SLA{sl}{6}}{f_{(4,1,1)}}$ and $\uQHR{\kk}{\SLA{sl}{6}}{f_{(3,3)}}$) are not related by the partial ordering however.}
\end{figure}
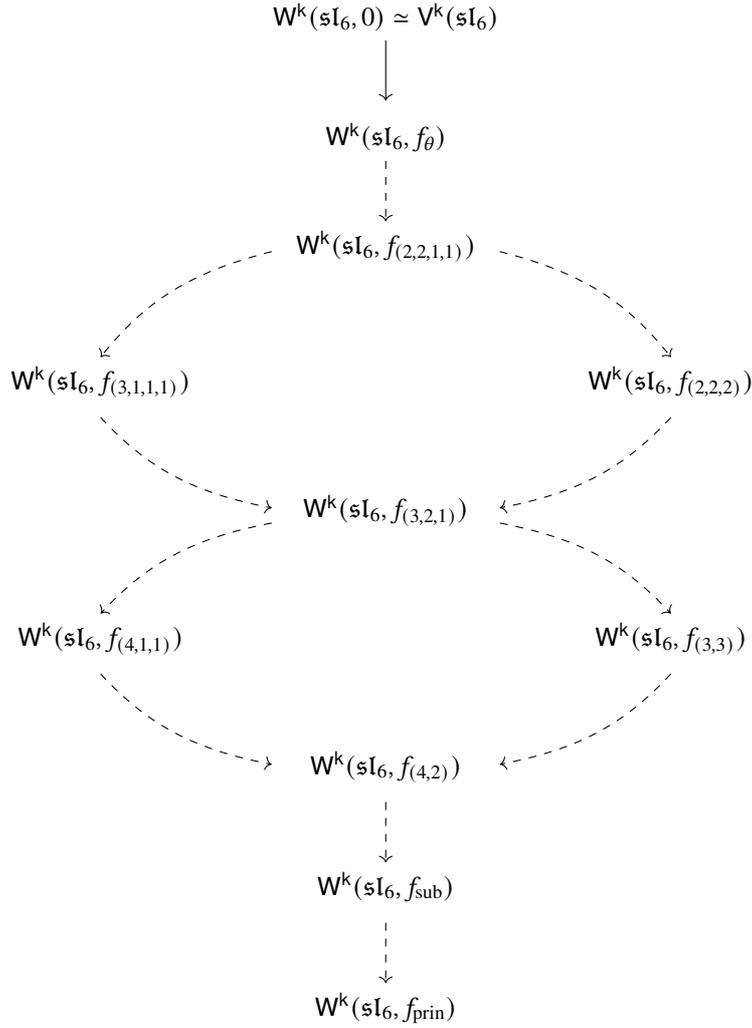

All of these distinguished orbits correspond to partitions of $(n+1)$ whose Young diagrams consist of a single hook; these are partitions of the form $(n-m+2,1^{m-1})$ for some $m = 1, \dots,  n+1$. We call such a nilpotent element \emph{hook-type}. Representatives for the hook-type nilpotent orbits are given by
\begin{equation} \label{eq:hooknilpotent}
    \fhook{n+1} = 0, \qquad \fhook{m} = \sum_{i=m}^{n} M_{i+1,i}
\end{equation}
for $m=1,\dots,n$, where $M_{i,j}$ is the $(i,j)$ elementary matrix. Of course all conjugates of the above matrices are representatives of the same nilpotent orbit. 

\begin{definition} \label{def:hook}
Let $\kk \neq -\dcox$ and $m = 1, \dots, n+1$. The  (universal) $m$'th \emph{hook-type $\slnpone$ W-algebra} is the W-algebra $\uQHR{\kk}{\slnpone}{\fhook{m}}$.
\end{definition}

Following the description in \cite{CreTri22}, for $m\geq 2$, $\uQHR{\kk}{\slnpone}{\fhook{m}}$ has an affine vertex subalgebra isomorphic to
\begin{equation}
    \begin{cases}
        \VOA{H}, & m=2,\\
        \uaffvoa{\kk+n+1-m}{\mathfrak{gl}_{m-1}} = \VOA{H} \otimes \uaffvoa{\kk+n+1-m}{\mathfrak{sl}_{m-1}}, & m>2,
    \end{cases}
\end{equation}
where $\VOA{H}$ is a rank one Heisenberg vertex algebra. The hook-type W-algebra $\uQHR{\kk}{\slnpone}{\fhook{m}}$ is strongly generated by the fields from this subalgebra, along with fields $L(z), X^3(z),\dots,X^{n-m+2}(z)$ of conformal dimension $2,3, \dots, n-m+2$ respectively and $2(m-1)$ fields $P^{\pm, 1}(z), \dots, P^{\pm, m-1}(z)$ all of conformal dimension $\frac{n-m+3}{2}$. Here, $L(z)$ is a conformal field with central charge
\begin{align} \label{eq:hookcc}
    \cc^{(n-m+2,1^{m-1})} =&\ \frac{\kk n (n+2)}{\kk+n+1}-\kk (-m+n+2) \left((-m+n+2)^2-1\right)\\
        &+(m-n-1) \left(m^2 (n+2)-2 m \left(n^2+3 n+3\right)+n^3+4 n^2+5 n+3\right). \notag
\end{align}

General formulae for the full operator product algebra $\uQHR{\kk}{\slnpone}{\fhook{m}}$ are not known. However quantum hamiltonian reduction does provide us with some information \cite{KacQua04}. For example, the fields $L(z), X^3(z),\dots,X^{n-m+2}(z)$ are $\mathfrak{gl}_{m-1}$-invariant, while the fields $P^{+, 1}(z), \dots, P^{+, m-1}(z)$ transform as $\CC^{m-1}$ (the standard $\mathfrak{gl}_{m-1}$-module) and the fields $P^{-, 1}(z), \dots, P^{-, m-1}(z)$ transform as the dual of $\CC^{m-1}$. Moreover, by \cite[Lem.~3.4]{CreTri22}, the generator $J(z)$ of the Heisenberg vertex algebra $\VOA{H}$ can be normalised so that 
\begin{equation} \label{eq:JPOPE}
    J(z) P^{\pm, i}(w) \sim \frac{\pm  P^{\pm, i}(w)}{z-w}, \qquad J(z) J(w) \sim \frac{-(m-1) (1 + n - (\kk+n+1) (n- m +2)) \wun(w)}{(n + 1)(z-w)^2}.
\end{equation}

Outside of specific choices of $n$ and $m$, much about $\uQHR{\kk}{\slnpone}{\fhook{m}}$ is still a mystery. However a quick computation shows that $(n-m+2,1^{m-1}) \ge (n-m'+2,1^{m'-1})$ when $m<m'$. So we expect that there is an inverse reduction
\begin{equation} \label{eq:invembHook}
    \uQHR{\kk}{\slnpone}{\fhook{m'}} \hookrightarrow \uQHR{\kk}{\slnpone}{\fhook{m}} \otimes \VOA{V}
\end{equation}
for some vertex operator algebra $\VOA{V}$. Indeed there is a partial reduction amongst finite W-algebras corresponding to hook-type nilpotents in $\slnpone$ \cite[Prop.~4.1.1]{GenStage22}. Additionally, hook-type $\slnpone$ W-algebras form a path in the partial ordering of W-algebras including both extremes: the affine vertex operator algebra $\uaffvoa{\kk}{\slnpone}$ and the principal W-algebra $\uQHR{\kk}{\slnpone}{\fprin}$. 

Understanding inverse reduction for hook-type $\slnpone$ W-algebras is therefore an ideal test case for the inverse reduction approach as a whole. The existence of a path of this kind also has the potential to greatly assist in studying the representation theory of all W-algebras along it and their simple quotients, including $\uaffvoa{\kk}{\slnpone}$ and $\saffvoa{\kk}{\slnpone}$ (the latter likely at certain levels $\kk$).

\section{Wakimoto realisation of \texorpdfstring{$\uQHR{\kk}{\slnpone}{\fhook{m}}$}{Wk(sln+1,fm)}} \label{sec:WakWalg}

Inverse reduction embeddings for small $n$ and $m$ can be obtained by fairly direct methods; In these cases, the operator product expansions are known on both sides and are straightforward to work with. The desired embedding can be obtained by making some reasonable assumptions, imposing the relevant operator product expansions and checking for injectivity. 

For larger $n$ and $m$, the complexity of the operator product expansions makes this approach exceedingly difficult. Therefore to prove the existence of this embedding in general, we need some more information about $\uQHR{\kk}{\slnpone}{\fhook{m}}$. Following the approach described in \cite{Feh21c} for the inverse reduction between the principal ($m=1$ hook-type) $\slnpone$ W-algebra to the subregular ($m=2$ hook-type) W-algebra, we are led to consider free-field realisations of $\uQHR{\kk}{\slnpone}{\fhook{m}}$ obtained as the kernel of certain screening operators \cite{Gen17}.

\subsection{Wakimoto realisation of \texorpdfstring{$\uaffvoa{\kk}{\slnpone}$}{Vk(sln+1)}} \label{subsec:WakAff}
To describe the Wakimoto realisation of the universal affine vertex operator algebra $\uaffvoa{\kk}{\slnpone}$ ($m=n+1$ hooktype $\slnpone$ W-algebra), we first need to specify some notation. A convenient basis for $\slnpone$ is given by the $(n+1) \times (n+1)$ elementary matrices $M_{i,j}$ with $i\neq j$ along with the diagonal matrices $M_{i,i} - M_{i+1,i+1}$ where $i=1, \dots, n$. The latter matrices form a basis for a Cartan subalgebra $\cartan$ of $\slnpone$. Denote the simple roots of $\slnpone$ by $\sroots = \set{\sroot{1},\dots, \sroot{n}}$. The set of positive roots, denoted by $\proots$, consists of sums
\begin{equation}
    \sroot{i,j} = \sroot{i} + \sroot{i+1} + \dots + \sroot{j},
\end{equation}
where $1 \leq i \leq j \leq n$. The aforementioned basis of $\slnpone$ is in fact a Cartan-Weyl basis under the identifications
\begin{equation}
    \prvec{i,j} = \prvec{\sroot{i,j}} = M_{i,j+1}, \qquad 
    \cvec{i} = \cvec{\sroot{i}} = M_{i,i}-M_{i+1,i+1}, \qquad
    \nrvec{i,j} = \nrvec{\sroot{i,j}} = M_{j+1,i}.
    \end{equation}

Associated to the Cartan subalgebra $\cartan=\text{span}\set{\cvec{1},\dots,\cvec{n}}$ of $\slnpone$ is the Heisenberg vertex algebra $\uaffvoa{\kk+\dcox}{\cartan}$ where $\dcox=n+1$ is the dual Coxeter number of $\slnpone$. Strong generating fields of $\uaffvoa{\kk+\dcox}{\cartan}$ are denoted by $\alpha_1(z), \dots, \alpha_n(z)$, and these fields satisfy the operator product expansions
\begin{equation}
    \alpha_i(z) \alpha_j(w) \sim \frac{(\kk+\dcox)\inner{\sroot{i}}{\sroot{j}} \wun(w)}{(z-w)^2}.
\end{equation}

To describe the Wakimoto realisation of $\uaffvoa{\kk}{\slnpone}$, we need both the  Heisenberg vertex algebra $\uaffvoa{\kk+\dcox}{\cartan}$ and a bosonic ghost system $\bgvoa_{\alpha} \cong \bgvoa(\CC^2)$ for each positive root $\alpha \in \proots$. Denote the strong generating fields of $\bgvoa_{\alpha}$ by $\beta_{\alpha}(z)$ and $\gamma_{\alpha}(z)$, whose only singular operator product expansion is
\begin{equation}
    \beta_\alpha(z) \gamma_\alpha(w) \sim \frac{-\wun(w)}{z-w}.
\end{equation}
To save ink, we will frequently use the notation $\beta_{i,j}(z)=\beta_{\sroot{i,j}}(z)$ and $\gamma_{i,j}(z)=\gamma_{\sroot{i,j}}(z)$. The ghost fields corresponding to simple roots will also occasionally be denoted $\beta_i(z) = \beta_{i,i}(z)$ and $\gamma_i(z) = \gamma_{i,i}(z)$. 

The Wakimoto realisation of $\uaffvoa{\kk}{\slnpone}$ is an embedding of vertex operator algebras 
\begin{equation}
    \uaffvoa{\kk}{\slnpone} \hookrightarrow \uaffvoa{\kk+\dcox}{\cartan} \otimes \bigotimes_{\alpha \in \proots} \bgvoa_\alpha.
\end{equation}

General formulas for this embedding are known but are difficult to work with \cite{Fre05Opers,Pet97Free,KuwaConf90}. When $\kk$ is generic, the image of the embedding can alternatively be described in terms of $n$ screening operators $\int S_i(z) \dd z$ acting on $\uaffvoa{\kk+\dcox}{\cartan} \otimes \bigotimes_{\alpha \in \proots} \bgvoa_\alpha$, where the fields $S_i(z)$ involve intertwining operators of Fock modules for the free fields. 

Following \cite{Fre05Opers}, let $N_+$ be the Lie subgroup of $\grp{SL}(n+1)$ corresponding to the upper nilpotent subalgebra $\mathfrak{n}_+$ of $\slnpone$. That is, $N_+$ consists of all matrices of the form
\begin{equation} \label{eq:bigmatrix}
\begin{bmatrix}
    1 & x_{1,1} & x_{1,2} & \cdots & x_{1,n-1} & x_{1,n}\\
    0 & 1 & x_{2,2} & & & x_{2,n} \\
    \vdots& & & \ddots & & \vdots \\
    & & & &  x_{n-1,n-1} & x_{n-1,n} \\
    0 &  & & &   1 & x_{n,n}  \\
    0 & 0  &   & \cdots &  0 & 1 \\
\end{bmatrix},
\end{equation}
where $x_{i,j} \in \CC$. The right action of $N_+$ on $N_+$ by matrix multiplication induces an anti-homomorphism $\rho^R: \mathfrak{n}_+ \rightarrow \mathcal{D}(N_+)$ where $\mathcal{D}(N_+)$ is the ring of polynomial differential operators on $N_+$. We can write  \cite[Sec.~1.5]{Fre05Opers}
\begin{equation} \label{eq:PFunctions}
    \rho^R(\prvec{i}) =\sum_{j\leq k} P^{R,\set{j,k}}_{i} \left(x\right) \frac{\pd}{\pd x_{j,k}}
\end{equation}
for some polynomials $P^{R,\set{j,k}}_{i} \left(x\right) \in \CC \sqbrac{x_{1,1},\dots}$. 
\begin{lemma}\label{lem:WakScreeners}
Let $i \in \set{1,\dots,n}$. Then,
\begin{equation}
    \rho^R(\prvec{i}) = \frac{\pd}{\pd x_{i,i}} + \sum_{j=1}^{i-1} x_{i-j,i-1} \frac{\pd}{\pd x_{i-j,i}}.
\end{equation}
\end{lemma}
\begin{proof}
Write the matrix $X$ in \eqref{eq:bigmatrix} as 
\begin{equation}
    X = I + \sum_{i<j} x_{i,j-1} M_{i,j}.
\end{equation}
The right action of $\prvec{i}=M_{i,i+1}$ is then given by
\begin{align}
    \rho^R(\prvec{i}) \left( X \right) &= X \prvec{i} = \left( I + \sum_{j<k} x_{j,k-1} M_{j,k} \right) M_{i,i+1} = M_{i,i+1} + \sum_{j=1}^{i-1} x_{j,i-1} M_{j,i+1} \\
    &= \frac{\pd}{\pd x_{i,i}} \left( X \right) + \sum_{j=1}^{i-1} x_{j,i-1} \frac{\pd}{\pd x_{j,i}} \left( X \right ) = \left(\frac{\pd}{\pd x_{i,i}} + \sum_{j=1}^{i-1} x_{i-j,i-1} \frac{\pd}{\pd x_{i-j,i}} \right) X. \notag \qedhere
\end{align}
\end{proof}

The screening fields $S_i(z)$, $i\in \set{1,\dots,n}$, of the Wakimoto realisation of $\uaffvoa{\kk}{\slnpone}$ are obtained by replacing the operators $\frac{\pd}{\pd x_{j,k}}$ and $x_{j,k}$ in $ \rho^R(\prvec{i})$ with the fields $\beta_{j,k}(z)$ and  $\gamma_{j,k}(z)$ respectively and taking the normally-ordered product of the result with a vertex operator for the Heisenberg vertex algebra $\uaffvoa{\kk+\dcox}{\cartan}$. The latter are intertwining operators of certain Fock modules for $\uaffvoa{\kk+\dcox}{\cartan}$ (see \cite[Sec.~5.2]{FBZ04}). The operator product expansion between a strong generator of $\uaffvoa{\kk+\dcox}{\cartan}$ and a vertex operator $\ee^{\lambda}(z)$ is given by
\begin{equation}
    \alpha_i(z) \ee^{\lambda}(w) \sim \frac{\lambda(h_i)\ee^{\lambda}(w)}{z-w}.
\end{equation}

The vertex operator needed for $S_i(z)$ is $\ee^{\frac{-1}{\kk+\dcox} \alpha_i}(z)$. Concretely, and in view of \cref{lem:WakScreeners},
\begin{equation} \label{eq:WakScreeners}
    S_i(z) = \no{\left( \beta_i(z) + \sum_{j=1}^{i-1} \gamma_{i-j,i-1}(z) \beta_{i-j,i}(z)  \right) \ee^{\frac{-1}{\kk+n+1} \alpha_i}(z) }.
\end{equation}

For example,
\begin{equation} \label{eq:WakScreenersExp}
\begin{gathered}
    S_1(z) = \no{\beta_1(z) \ee^{\frac{-1}{\kk+n+1} \alpha_1}(z)},\\
    S_2(z) = \no{\big(\beta_2(z)+\gamma_1(z) \beta_{1,2}(z) \big) \ee^{\frac{-1}{\kk+n+1} \alpha_2}(z)},\\
    S_3(z) = \no{\big(\beta_3(z) + \gamma_2(z) \beta_{2,3}(z) + \gamma_{1,2}(z) \beta_{1,3}(z)\big) \ee^{\frac{-1}{\kk+n+1} \alpha_3}(z)}.
\end{gathered}
\end{equation}
For generic $\kk$, the image of the embedding $\uaffvoa{\kk}{\slnpone} \hookrightarrow  \uaffvoa{\kk+\dcox}{\cartan} \otimes \bigotimes_{\alpha \in \proots} \bgvoa_\alpha$ of the Wakimoto realisation is given by \cite[Prop.~8.2]{Fre05Opers}
\begin{equation} \label{eq:WakAffScreen}
    \uaffvoa{\kk}{\slnpone} \cong \bigcap_{i=1}^n \textup{ker} \int S_i(z) \ \dd z \subset \uaffvoa{\kk+\dcox}{\cartan} \otimes \bigotimes_{\alpha \in \proots} \bgvoa_\alpha.
\end{equation}

Choosing different coordinates on $N_+$ (as long as they are homogeneous in the sense defined in \cite{Fre05Opers}) might result in different expressions for $ \rho^R(\prvec{i})$ and therefore for the screening operators $S_i(z)$. However the image of the screening operators obtained from $S_i(z)$ is independent, up to isomorphism, of the choice of homogeneous coordinates. The Wakimoto realisation can be defined in a coordinate-independent way, but this level of generality is not required for our purposes.

\subsection{Free field realisation of \texorpdfstring{$\uQHR{\kk}{\slnpone}{\fhook{m}}$}{Wk(sln+1,fm)}} \label{subsec:WakWalg}

Choose a nilpotent element $f \in \slnpone$ and a semisimple $x \in \slnpone$ whose adjoint action on $\slnpone$ generates a good grading for $f$ (such a pair of $f$ and $x$ exists for any nilpotent orbit in $\slnpone$ by \cite[Thm.~4.1]{ElasClass05}). One can then partition the set of simple roots $\sroots$ according to
\begin{equation}
    \sroots = \sroots_0 \cup \sroots_{\frac{1}{2}} \cup \sroots_{1}, \quad \sroots_a = \set{\alpha \in \sroots \ \vert \ (\slnpone)_\alpha \subset (\slnpone)_{a}}.
\end{equation}
We likewise define $\roots_a$ and $(\proots)_a$ in the obvious way. Let $\bgvoa\left((\slnpone)_{1/2}\right)$ be the neutral ghost vertex algebra from \cref{subsec:QHR}. While it is sometimes appropriate to consider gradings that necessitate the inclusion of $\bgvoa\left((\slnpone)_{1/2}\right)$, it is often convenient to choose our grading such that $(\slnpone)_{1/2}$ is zero, that is an \emph{even} grading of $\slnpone$. Given a nilpotent orbit, one can always pick a representative $f$ for which an even good grading of $\slnpone$ exists \cite[Prop.~4.2]{ElasClass05}.

The Wakimoto realisation of $\uQHR{\kk}{\slnpone}{f}$ is an embedding  \cite{GenScreen20}, for $\kk \neq -\dcox$,
\begin{equation}
    \uQHR{\kk}{\slnpone}{f} \hookrightarrow \uaffvoa{\kk+\dcox}{\cartan}\otimes \bgvoa\left((\slnpone)_{\frac{1}{2}}\right) \otimes \bigotimes_{\alpha \in (\proots)_0} \bgvoa_\alpha.
\end{equation}

We call a level $\kk$ \emph{generic} (for a given $f \in \mathfrak{g}$) if the homology of certain sub-complexes of the quantum hamiltonian reduction complex for $f \in \mathfrak{g}$ are isomorphic (see \cite[Def.~4.5]{Gen17}). It is known that the set of generic levels is Zariski dense in $\CC$ \cite[Lem.~4.4]{Gen17}.

When $\kk$ is generic, this embedding can be described as the intersection of kernels of screening operators obtainable from screening operators for the Wakimoto realisation of $\uaffvoa{\kk}{\slnpone}$ coming from particular choices of coordinates in $N_+$. Not all choices of coordinates for $N_+$ are suitable for all $\slnpone$ W-algebras: Following \cite[Sec.~4.1]{GenScreen20}, consider the subalgebras 
\begin{equation}
    (\slnpone)_{>0} = \bigoplus_{j>0} (\slnpone)_j, \quad (\slnpone)^+_0 = (\slnpone)_0 \cap \bigoplus_{\alpha \in \proots} (\slnpone)_\alpha,
\end{equation}
and their unipotent images $G_{>0}$ and $G^+_0$ in $N_+$ under the exponential map $\exp:\slnpone \rightarrow \SLG{SL}{n+1}$. Any choice of homogeneous coordinates $c(\mathfrak{n}_+)$ on $N_+$ gives a Wakimoto realisation of $\uaffvoa{\kk}{\slnpone}$ but will not necessarily interact nicely with the homology defining the W-algebra $\uQHR{\kk}{\slnpone}{f}$. 

To ensure a nice interaction, choose homogeneous coordinates $c(\mathfrak{g}_{>})$ and $c(\mathfrak{g}^+_{0})$ on $G_{>0}$ and $G^+_0$ respectively. Then, since $G_{>0} \times G^+_0 \cong N_+$, we obtain coordinates $c(\mathfrak{n}_+) = c(\mathfrak{g}_{>}) \cdot c(\mathfrak{g}^+_{0})$. Let $S_i(z)$ be the screening operators for the Wakimoto realisation of $\uaffvoa{\kk}{\slnpone}$ corresponding to this choice of coordinates. Then the screening fields $Q_i(z)$ for the W-algebra $\uQHR{\kk}{\slnpone}{f}$ are \cite[Thm.~4.8]{GenScreen20}
\begin{equation} \label{eq:WalgWakScreenings}
\begin{gathered}
    Q_i(z) = 
    \begin{cases}
        S_i(z) & \sroot{i} \in \sroots_{0}, \\
        \no{\sum_{\alpha_{j,k} \in \sqbrac{\sroot{i}}} P^{R,\set{j,k}}_{i} \left(\gamma(z)\right) \delta^{\alpha_{j,k}}(z) \ee^{\frac{-1}{\kk+\dcox} \alpha_{i}}(z)} & \sroot{i} \in \sroots_{\frac{1}{2}}, \\
        \no{\sum_{\alpha_{j,k} \in \sqbrac{\sroot{i}}} P^{R,\set{j,k}}_{i} \left(\gamma(z)\right)  \inner{f}{e_{\alpha_{j,k}}} \ee^{\frac{-1}{\kk+\dcox} \alpha_{i}}(z)} & \sroot{i} \in \sroots_{1},
    \end{cases}
\end{gathered}   
\end{equation}
where $P^{R,\set{j,k}}_{i} \left(\gamma(z)\right)$ is obtained from $P^{R,\set{j,k}}_{i} \left(x\right)$ in \eqref{eq:PFunctions} by replacing $x=(x_{1,1},\dots)$ with $\gamma(z)=(\gamma_{1,1}(z),\dots)$ and
\begin{equation}
    \sqbrac{\sroot{i}} = \left\{ \alpha_j \in \proots \ \vert \ \alpha_j-\sroot{i} \in \bigoplus_{\gamma \in \sroots_0} \ZZ \gamma \right\}.
\end{equation}

For example, let $\fhook{m}$ be the hook-type nilpotent element defined in \eqref{eq:hooknilpotent}. A good grading for $\fhook{m}$ is defined by
\begin{equation} \label{eq:GoodGrading}
\begin{tikzpicture}
    \node[circle,draw,fill=black,label=below:$\sroot{1}$,label=above:$0$] (1) at (0,0) {};
    \node[circle,draw,fill=black,label=below:$\sroot{2}$,label=above:$0$] (2) at (2,0) {};
    \node (A) at (3.5,0) {\dots};
    \node[circle,draw,fill=black,label=below:$\sroot{m-1}$,label=above:$0$] (m-1) at (5,0) {};
    \node[circle,draw,fill=black,label=below:$\sroot{m}$,label=above:$1$] (m) at (7,0) {};
    \node (B) at (8.5,0) {\dots};
    \node[circle,draw,fill=black,label=below:$\sroot{n}$,label=above:$1$] (n) at (10,0) {};

    \draw (1)--(2);
    \draw (2)--(A);
    \draw (A)--(m-1);
    \draw (m-1)--(m);
    \draw (m)--(B);
    \draw (B)--(n);
\end{tikzpicture} 
\end{equation}
where the grade given to the positive root vector $\prvec{\sroot{i}}$ is given by the number above $\sroot{i}$. It is easy to see that $\sroots_0 = \set{\sroot{1}, \dots, \sroot{m-1}}$ and $(\proots)_0 = \set{\sroot{i,j} \ \vert \ j \leq m-1}$. Hence the Wakimoto realisation of the $m$'th hook-type $\slnpone$ W-algebra is an embedding
\begin{equation} \label{eq:WakEmbHook}
    \uQHR{\kk}{\slnpone}{\fhook{m}} \hookrightarrow \uaffvoa{\kk+\dcox}{\cartan} \otimes \bigotimes_{j=1}^{m-1}\bigotimes_{i=1}^{j} \bgvoa_{\sroot{i,j}}.
\end{equation}

To obtain the screening operators that describe the embedding \eqref{eq:WakEmbHook} at generic levels, we first must check that the choice of coordinates for $X$ in \eqref{eq:bigmatrix} is compatible with the choice of good grading in \eqref{eq:GoodGrading}. By definition,
\begin{equation}
    (\slnpone)_{>0} = \text{span} \set{\prvec{i,j} \ \vert \ m \leq j}, \qquad (\slnpone)^+_0 = \text{span} \set{\prvec{i,j} \ \vert \ j \leq m-1}.
\end{equation}
A direct calculation shows that
\begin{equation}
    X = \left(\prod_{i \leq n} \exp(x_{i,n} e_{i,n}) \right)\cdot \left(\prod_{i \leq n-1} \exp(x_{i,n-1} e_{i,n-1}) \right) \cdot \cdot \cdot \left(\exp(x_{1,2} e_{1,2}) \cdot \exp(x_{2,2} e_{2,2}) \right) \cdot \left(\exp(x_{1,1} e_{1,1})\right),
\end{equation}
where the products over $i$ above are ordered left-to-right with increasing $i$. Choosing coordinates on $(\slnpone)^+_0$ and $(\slnpone)_{>0}$ according to
\begin{equation}
    X_m = \left(\prod_{i \leq m-1} \exp(x_{i,m-1} e_{i,m-1}) \right) \cdot \cdot \cdot \left(\exp(x_{1,2} e_{1,2}) \cdot \exp(x_{2,2} e_{2,2}) \right) \cdot \left(\exp(x_{1,1} e_{1,1})\right)
\end{equation}
and
\begin{equation}
    Y_m = \left(\prod_{i \leq n} \exp(x_{i,n} e_{i,n}) \right)  \cdot \cdot \cdot \left(\prod_{i \leq m} \exp(x_{i,m} e_{i,m}) \right)
\end{equation}
respectively, we have that $X = Y_m \cdot X_m$. So the coordinate system defined in \eqref{eq:bigmatrix} is compatible with $G_{>0} \times G^+_0 \cong N_+$ for the grading defined by \eqref{eq:GoodGrading}.

Now observe that since $\sroot{i} \in \sroots_0$ for all $i \leq m-1$, $Q_i(z) = S_i(z)$ for such $i$. On the other hand, $\sroot{i} \in \sroots_1$ for all $i \geq m$. For such $i$, $\sqbrac{\sroot{i}} \setminus \set{\sroot{i}}$ is only nonempty when $i=m$ and
\begin{equation}
     \sqbrac{\sroot{m}} = \left\{ \alpha_{j,m} \in \proots \ \vert \ j=1, \dots, m \right\}.
\end{equation}

So for $i > m$, $\sqbrac{\sroot{i}} = \set{\sroot{i}}$ and
\begin{equation}
    Q_i(z) =  \no{ P^{R,\set{i,i}}_{i} \left(\gamma(z)\right)  \inner{\fhook{m}}{e_{\alpha_{i}}} \ee^{\frac{-1}{\kk+\dcox} \alpha_{i}}(z)} 
    = \ee^{\frac{-1}{\kk+\dcox} \alpha_{i}}(z)
\end{equation}
and, using that $\inner{\fhook{m}}{e_{\alpha_{j,m}}} = \delta_{j,m}$,
\begin{align}
    Q_m(z) &= \no{\sum_{\alpha_{j,m} \in \sqbrac{\sroot{i}}} P^{R,\set{j,m}}_{i} \left(\gamma(z)\right)  \inner{\fhook{m}}{e_{\alpha_{j,m}}} \ee^{\frac{-1}{\kk+\dcox} \alpha_{m}}(z)} \\
    &= \no{P^{R,\set{m,m}}_{m} \left(\gamma(z)\right)  \inner{\fhook{m}}{e_{\alpha_{m,m}}} \ee^{\frac{-1}{\kk+\dcox} \alpha_{m}}(z)} \notag\\
    &= \ee^{\frac{-1}{\kk+\dcox} \alpha_{m}}(z). \notag
\end{align}
Therefore, for generic $\kk$, the image of the embedding \eqref{eq:WakEmbHook} of the Wakimoto realisation of $\uQHR{\kk}{\slnpone}{\fhook{m}}$ is 
\begin{equation} \label{eq:WakHookScreen}
    \uQHR{\kk}{\slnpone}{\fhook{m}} \cong \left( \bigcap_{i=1}^{m-1} \textup{ker} \int S_i(z) \ \dd z \right) \cap \left(  \bigcap_{i=m}^{n} \textup{ker} \int \ee^{\frac{-1}{\kk+\dcox} \alpha_{i}}(z) \ \dd z \right) \subset \uaffvoa{\kk+\dcox}{\cartan} \otimes \bigotimes_{j=1}^{m-1}\bigotimes_{i=1}^{j} \bgvoa_{\sroot{i,j}}.
\end{equation}

One should in principle check that the ghost parts of the screening fields $S_i(z)$ with $i \leq m-1$ contain only fields from $\bigotimes_{j=1}^{m-1}\bigotimes_{i=1}^{j} \bgvoa_{\sroot{i,j}}$, but this is an immediate consequence of \cref{lem:WakScreeners} due to the consequent formula \eqref{eq:WakScreeners}.

In what follows, it will be useful to have an additional free-field realisation of $\uQHR{\kk}{\slnpone}{\fhook{m}}$. As the W-algebra $\uQHR{\kk}{\slnpone}{f}$ only depends on the nilpotent orbit of $\slnpone$ containing $f$ up to isomorphism, we may choose any another hook-type nilpotent element $\bar{f}^{(m)}$ and the resulting W-algebra  $\uQHR{\kk}{\slnpone}{\bar{f}^{(m)}}$ will be isomorphic to $\uQHR{\kk}{\slnpone}{\fhook{m}}$. 

Moreover, due to the $f$ dependence of \eqref{eq:WalgWakScreenings}, different choices of representatives of the  $m$'th hook-type nilpotent orbit of $\slnpone$ will lead to (in general) different sets of screening operators that different describe subalgebras of $\uaffvoa{\kk+\dcox}{\cartan} \otimes \bigotimes_{j=1}^{m-1}\bigotimes_{i=1}^{j} \bgvoa_{\sroot{i,j}}$ isomorphic to $\uQHR{\kk}{\slnpone}{\fhook{m}}$. 

Consider the nilpotent matrix
\begin{equation}
    \bar{f}^{(m)} = \nrvec{1,m} + \sum_{i=m+1}^n \nrvec{i}.
\end{equation}
This matrix is conjugate to $\fhook{m}$ and so represents the same nilpotent orbit in $\slnpone$. Moreover, the grading defined by \eqref{eq:GoodGrading} is a good grading for $\bar{f}^{(m)}$. Hence the Wakimoto realisation of $\uQHR{\kk}{\slnpone}{\bar{f}^{(m)}}$ is again an embedding into $\uaffvoa{\kk+\dcox}{\cartan} \otimes \bigotimes_{j=1}^{m-1}\bigotimes_{i=1}^{j} \bgvoa_{\sroot{i,j}}$. To describe the corresponding screening operators, recall that the only place where the particular choice of $f$ was used in calculating the screening operators for the Wakimoto realisation of $\uQHR{\kk}{\slnpone}{\fhook{m}}$ was in $Q_i(z)$ for $i \geq m$. 

For $i > m$, the screening operator $\bar{Q}_i(z)$ for $\uQHR{\kk}{\slnpone}{\bar{f}^{(m)}}$ is equal to $Q_i(z)$ since $\inner{\bar{f}^{(m)}}{\prvec{i}}=\inner{\fhook{m}}{\prvec{i}}$ for such $i$. In fact, the only screening operator for $\uQHR{\kk}{\slnpone}{\bar{f}^{(m)}}$ that differs from those of $\uQHR{\kk}{\slnpone}{\fhook{m}}$ is the $m$'th one, which is
\begin{align}
    \bar{Q}_m(z) &= \no{\sum_{\alpha_{j,m} \in \sqbrac{\sroot{m}}} P^{R,\set{j,m}}_{i} \left(\gamma_{l,l'}(z)\right)  \inner{f}{e_{\alpha_{j,m}}} \ee^{\frac{-1}{\kk+\dcox} \alpha_{m}}(z)} \\
    &= \no{P^{R,\set{1,m}}_{m} \left(\gamma_{l,l'}(z)\right)  \inner{\nrvec{1,m}}{\prvec{j,m}} \ee^{\frac{-1}{\kk+\dcox} \alpha_{m}}(z)} \notag\\
    &= \no{\gamma_{1,m-1}(z) \ee^{\frac{-1}{\kk+\dcox} \alpha_{m}}(z)}, \notag
\end{align}
where we have used $\inner{\bar{f}^{(m)}}{e_{\alpha_{j,m}}} = \inner{\nrvec{1,m}}{\prvec{j,m}} = \delta_{j,1}$ and \eqref{eq:WakScreeners}. Of course many other choices of hook-type nilpotents are possible, but only these two will be used in what follows.

\section{Inverse reduction} \label{sec:Inverse}
Continuing in the template of \cite{Feh21c}, the key step in proving the existence of an inverse reduction embedding is by bosonising some number of bosonic ghost systems and comparing the resulting screening operators of the involved W-algebras. As we will see, the inverse reduction from $\uQHR{\kk}{\slnpone}{\fhook{m-1}}$ to $\uQHR{\kk}{\slnpone}{\fhook{m}}$ only requires a single bosonisation. All other inverse reductions amongst hook-type $\slnpone$ W-algebras are then obtained as compositions of these ``one step" inverse reductions.

\subsection{From minimal to affine} \label{subsec:MintoAff}
Before describing the general hook-to-hook inverse reduction, we will explore the minimal-to-affine inverse reduction in some detail. The reasons are twofold. For one, the affine vertex algebra $\uaffvoa{\kk}{\slnpone}$ and its representation theory (as well as that of its simple quotients) are of considerable interest. Secondly, as we will see, the hook-to-hook inverse reduction has a lot of similarities to a minimal-to-affine reduction, but for $\mathfrak{sl}_m$ for some $m < n+1$. So the set up used in the minimal-to-affine case can be adapted to the general hook-to-hook case with a few modifications. 

Recall from \cref{sec:WakWalg} the Wakimoto realisations and associated screening operators of $\uaffvoa{\kk}{\slnpone}$ and the minimal $\slnpone$ W-algebra $\uQHR{\kk}{\slnpone}{\fmin} =\uQHR{\kk}{\slnpone}{\bar{f}^{(n)}}$:
\begin{equation} \label{eq:WakAffandMin}
\begin{aligned}
    \uaffvoa{\kk}{\slnpone} &\cong \bigcap_{i=1}^n \textup{ker} \int S_i(z) \ \dd z \subset \uaffvoa{\kk+\dcox}{\cartan} \otimes \bigotimes_{\alpha \in \proots} \bgvoa_\alpha, \\
    \uQHR{\kk}{\slnpone}{\fmin} &\cong \left( \bigcap_{i=1}^{n-1} \textup{ker} \int S_i(z) \ \dd z \right) \cap \textup{ker} \int \no{\gamma_{1,n-1}(z) \ee^{\frac{-1}{\kk+\dcox} \alpha_{n}}(z)} \ \dd z \subset \uaffvoa{\kk+\dcox}{\cartan} \otimes \bigotimes_{j=1}^{n-1}\bigotimes_{i=1}^{j} \bgvoa_{\sroot{i,j}}.
\end{aligned}
\end{equation}
The only screening operator that $\uaffvoa{\kk}{\slnpone}$ and $\uQHR{\kk}{\slnpone}{\fmin}$ do not superficially have in common is the $n$'th one. By \eqref{eq:WakScreeners},
\begin{equation}
    S_n(z) = \no{\left( \beta_n(z) + \sum_{j=1}^{n-1} \gamma_{n-j,n-1}(z) \beta_{n-j,n}(z)  \right) \ee^{\frac{-1}{\kk+n+1} \alpha_n}(z) }.
\end{equation}

The $j=n-1$ term in the sum is $\gamma_{1,n-1}(z)\beta_{1,n}(z)$, in which we see the ghost field appearing in the $n$'th screening operator for $\uQHR{\kk}{\slnpone}{\fmin}$. Our goal then is to `free' $\gamma_{1,n-1}(z)$ in such a way that the screening operators $S_i(z)$, $i < n$, are unchanged. The first order of business is thus to deal with $\beta_{1,n}(z)$. This is done by composing the Wakimoto realisation of $\uaffvoa{\kk}{\slnpone}$ with the \emph{FMS bosonisation} \cite{FMS86} of the ghost system $\bgvoa_{\sroot{1,n}}$.

FMS bosonisation is an embedding of a ghost vertex algebra $\bgvoa$ into the half-lattice vertex algebra $\Pi$, which is a vertex algebra with strong generators $c(z), d(z)$ and $\ee^{m c}(z)$ for $m \in \ZZ$ and singular operator product expansions \cite{BerRep02}
\begin{equation}
    c(z) d(w) \cong \frac{2 \wun(w)}{(z-w)^2}, \qquad d(z) \ee^{m c}(w) \sim \frac{2m \ee^{m c}(w)}{z-w}.
\end{equation}
The embedding is given by
\begin{equation}
    \beta(z) \mapsto \ee^c(z), \quad \gamma(z) \mapsto \frac{1}{2}\no{\left(c(z)+d(z)\right) \ee^{-c}(z)},
\end{equation}
which conveniently can be described in terms of a screening operator:
\begin{equation}
    \bgvoa \cong \textup{ker} \int \ee^{\frac{1}{2}c + \frac{1}{2}d}(z) \ dz \subset \Pi.
\end{equation}

Consider again the Wakimoto realisation of $\uaffvoa{\kk}{\slnpone}$ when $n>1$, which is when $(\proots) \setminus \hroot$ is nonempty. The $n=1$ minimal-to-affine inverse reduction is known \cite{SemInv94,AdaRea17} and the argument is essentially the same as what we are about to describe. Applying FMS bosonisation to the ghost vertex algebra $\bgvoa_{\sroot{1,n}}$, we obtain an embedding of $\uaffvoa{\kk}{\slnpone}$ into $\uaffvoa{\kk+\dcox}{\cartan} \otimes \Pi \otimes \bigotimes_{\alpha \in \proots \setminus \hroot} \bgvoa_\alpha$ specified by the screening operators (see \cite[Sec.~3.1]{Feh21c})
\begin{equation} \label{eq:MinToAffIntermediate}
     \uaffvoa{\kk}{\slnpone} \cong \left( \bigcap_{i=1}^{n-1} \textup{ker} \int S_i(z) \ \dd z \right) \cap \textup{ker} \int \no{\tilde{\gamma}_{1,n-1}(z) \ee^{\frac{-1}{\kk+\dcox} \tsroot{n}}(z)} \ \dd z \cap  \textup{ker} \int \ee^{\frac{1}{2}c + \frac{1}{2}d}(z) \ dz,
\end{equation}
where $\tsroot{n}(z) = \sroot{n}(z)-(\kk+\dcox)c(z)$ and 
\begin{equation} \label{eq:FirstTildedField}
    \tilde{\gamma}_{1,n-1}(z) = \gamma_{1,n-1}(z) + \no{ \beta_n(z)\ee^{-c}(z) + \sum_{j=1}^{n-2} \gamma_{n-j,n-1}(z) \beta_{n-j,n}(z)\ee^{-c}(z)}.
\end{equation}

So we have an embedding $\uaffvoa{\kk}{\slnpone}$ described by screening operators that are superficially the same as those of the Wakimoto realisation of $\uQHR{\kk}{\slnpone}{\fmin}$, as well as a screening operator coming from the bosonisation. However the domains are different, so there is no guarantee that the first $n$ screening operators in \eqref{eq:MinToAffIntermediate} behave like the minimal $\slnpone$ W-algebra ones. 

To show that they in fact do,  we must identify the fields in $\uaffvoa{\kk+\dcox}{\cartan} \otimes \Pi \otimes \bigotimes_{\alpha \in \proots \setminus \hroot} \bgvoa_\alpha$ that have singular operator product expansions with the fields appearing in the first $n$ screening operators in \eqref{eq:MinToAffIntermediate} and show that the subalgebra spanned by these fields is isomorphic to $\uaffvoa{\kk+\dcox}{\cartan} \otimes \bigotimes_{j=1}^{n-1}\bigotimes_{i=1}^{j} \bgvoa_{\sroot{i,j}}$.

For $i=1,\dots, n$, define $\tsroot{i}(z) = \sroot{i}(z)-\delta_{i,n}(\kk+\dcox)c(z)$ and let $\widetilde{\uaffvoa{\kk+\dcox}{\cartan}}$ be the subalgebra of $\uaffvoa{\kk+\dcox}{\cartan} \otimes \Pi \otimes \bigotimes_{\alpha \in \proots \setminus \hroot} \bgvoa_\alpha$ generated by these fields. Likewise, let $\widetilde{\Pi}$ be the subalgebra generated by
\begin{equation} \label{eq:tildeHalfLat}
\begin{gathered}
    \tilde{c}(z) = c(z), \qquad 
    \ee^{m\tilde{c}}(z) = \ee^{mc}(z), \\
    \hspace{-15em} \tilde{d}(z) = d(z) 
    - (\kk+\dcox)\inner{\fwt{n}}{\fwt{n}} c(z)
    + 2 \fwt{n}(z) \\
    \hspace{4em}- \sum_{\substack{\alpha,\alpha' \in \proots \setminus \hroot \\ \alpha+\alpha'=\hroot}} \no{\beta_\alpha (z) \beta_{\alpha'}(z)\ee^{-c}(z)} 
    - \sum_{\substack{\alpha,\alpha',\alpha'' \in \proots \setminus \hroot \\ -\alpha+\alpha'+\alpha''=\hroot}} \no{\gamma_\alpha(z) \beta_{\alpha'} (z) \beta_{\alpha''}(z)\ee^{-c}(z)}.
\end{gathered}
\end{equation}
Finally for $\alpha \in \proots \setminus \hroot$, let $\widetilde{\bgvoa_\alpha}$ be the subalgebra of $\uaffvoa{\kk+\dcox}{\cartan} \otimes \Pi \otimes \bigotimes_{\alpha \in \proots \setminus \hroot} \bgvoa_\alpha$ generated by the fields
\begin{equation} \label{eq:tildeGhosts}
\begin{aligned}
    \tilde{\beta}_\alpha(z) &= \beta_\alpha(z)
    - \frac{1}{2}\sum_{\substack{\alpha',\alpha'' \in \proots \setminus \hroot \\ \alpha'+\alpha''=\hroot+\alpha}} \no{\beta_{\alpha'} (z) \beta_{\alpha''}(z)\ee^{-c}(z)}, \\
    \tilde{\gamma}_\alpha(z) &= \gamma_\alpha(z) 
    + \sum_{\substack{\alpha' \in \proots \setminus \hroot \\ \alpha'=\hroot-\alpha}}\no{\beta_{\alpha'}(z)\ee^{-c}(z)}
    + \sum_{\substack{\alpha'',\alpha''' \in \proots \setminus \hroot \\ -\alpha''+\alpha'''=\hroot-\alpha}}\no{\gamma_{\alpha''}(z)\beta_{\alpha'''}(z)\ee^{-c}(z)}.
\end{aligned}
\end{equation}

These formulae appear complicated but simplify considerably in practice. For example, the sum in the definition of $\tilde{\beta}_\alpha(z)$, if it is nonzero, has two summands that are equal. The first sum in the definition of $\tilde{\gamma}_\alpha(z)$ has either one or no nonzero summands (see \eqref{eq:tildeBsimpler}).

\begin{definition}
A positive $\slnpone$ root $\alpha \in \proots$ is \emph{internal} if, when expressed as a sum of simple roots $\sroot{i} \in \sroots$, the coefficients of $\alpha_1$ and $\alpha_n$ are zero. A positive $\slnpone$ root is \emph{exposed} if it is not internal. 
\end{definition}

\begin{proposition} \label{lem:InternalExposed}
\begin{itemize}
    \item If $\alpha \in \proots \setminus \hroot$ is exposed, then $\tilde{\beta}_\alpha(z) = \beta_\alpha(z)$.
    \item If $\alpha \in \proots \setminus \hroot$ is exposed, the contribution to the sum in $\tilde{\gamma}_\alpha(z)$ is only non-zero for exposed $\alpha',\alpha'''$ and internal $\alpha''$.
    \item If $\alpha \in \proots \setminus \hroot$ is internal, then $\tilde{\gamma}_\alpha(z) = \gamma_\alpha(z)$.
    \item If $\alpha \in \proots \setminus \hroot$ is internal, the contribution to the sum in $\tilde{\beta}_\alpha(z)$ is only non-zero for exposed $\alpha', \alpha''$.
    \item The contribution to the first sum in the definition of $\tilde{d}(z)$ is only non-zero for exposed $\alpha, \alpha' \in \proots \setminus \hroot$.
    \item The contribution to the second sum in the definition of $\tilde{d}(z)$ is only non-zero for internal $\alpha \in \proots \setminus \hroot$ and exposed $\alpha', \alpha''$.
\end{itemize} 
\end{proposition}
\begin{proof}
    If $\alpha \in \proots \setminus \hroot$ is exposed, it must be $\alpha = \sroot{1,j}$ or $\sroot{k,n}$ for $j\neq n$ and $k \neq 1$. Since $\theta = \sum_{i=1}^n \alpha_i$,
    \begin{equation}
        \hroot+\sroot{1,j} = 2\sroot{1} + \dots + 2 \sroot{j} + \sroot{j+1} + \dots + \sroot{n}, \quad \hroot+\sroot{k,n} = \sroot{1} + \dots + \sroot{k-1} + 2\sroot{k} + \dots + 2\sroot{n}.
    \end{equation}
    The only way two positive roots $\alpha',\alpha'' \in \proots$ can sum to the above root lattice elements is if one of them is $\hroot$ which is excluded in the sum defining $\tilde{\beta}_\alpha(z)$. Hence the sum is zero and $\tilde{\beta}_\alpha(z) = \beta_\alpha(z)$. On the other hand, the difference $\hroot-\alpha$ in both cases is an also exposed root. If there were two roots $\alpha'',\alpha''' \in \proots \setminus \hroot$ such that 
    \begin{equation}
        -\alpha''+\alpha'''=\hroot-\alpha = \sroot{j+1,n} \quad \text{ or } \quad \sroot{1,k-1}, 
    \end{equation}
    then $\alpha'''$ must be exposed and $\alpha''$ must be internal. This proves the second statement. The remaining statements are proved in a similar way. 
\end{proof}

\noindent In terms of the aforementioned parametrisation of $\slnpone$ positive roots, the definitions of the tilded ghost fields \eqref{eq:tildeGhosts} can be be written in a more explicit form. Let $\alpha$ be internal, so $\alpha =\sroot{i,j}$ for some $1<i\leq j < n$. Then
\begin{equation} \label{eq:tildeBsimpler}
    \tilde{\beta}_{i,j}(z) = \beta_{i,j}(z)
   = \beta_{i,j}(z) - \no{\beta_{1,j} (z) \beta_{i,n}(z)\ee^{-c}(z)}.
\end{equation}
Likewise, let $\alpha \in \proots \setminus \hroot$ be an external root of the form $\alpha =\sroot{1,j}$ for some $j < n$ or $\alpha = \sroot{j,n}$ for some $1<j$. Then,
\begin{equation}
\begin{gathered}
    \tilde{\gamma}_{1,j}(z) 
    = \gamma_{1,j}(z) + \no{\beta_{j+1,n}(z)\ee^{-c}(z)} + \sum_{k=2}^{j} \no{\gamma_{k,j}(z)\beta_{k,n}(z)\ee^{-c}(z)}, \\
    \tilde{\gamma}_{j,n}(z) 
    = \gamma_{j,n}(z) + \no{\beta_{1,j-1}(z)\ee^{-c}(z)} + \sum_{k=j}^{n-1} \no{\gamma_{j,k}(z)\beta_{1,k}(z)\ee^{-c}(z)}.
\end{gathered}
\end{equation}
The formula for $\tilde{\gamma}_{1,n-1}(z)$ reproduces the field seen in \eqref{eq:FirstTildedField} obtained by bosonising the Wakimoto realisation of $\uaffvoa{\kk}{\slnpone}$.

\begin{proposition} \label{prop:tildeGhostOPEs}
    Let $\alpha, \gamma \in  \proots \setminus \hroot$. Then,
    \begin{equation}
        \tilde{\beta}_\alpha(z) \tilde{\beta}_\lambda(w) \sim 0 \sim \tilde{\gamma}_\alpha(z) \tilde{\gamma}_\lambda(w), \qquad 
        \tilde{\beta}_\alpha(z) \tilde{\gamma}_\lambda(w) \sim \frac{-\delta_{\alpha=\lambda} \wun(w)}{z-w}.
    \end{equation}
\end{proposition}
\begin{proof}
    That $ \tilde{\beta}_\alpha(z) \tilde{\beta}_\lambda(w) \sim 0$ is clear from \eqref{eq:tildeGhosts}. Additionally, 
    \begin{align}
        \tilde{\beta}_\alpha(z) \tilde{\gamma}_\lambda(w) \sim& \ \beta_\alpha(z) \left(\gamma_\lambda(w) 
       + \sum_{\substack{\lambda'',\lambda''' \in \proots \setminus \hroot \\ -\lambda''+\lambda'''=\hroot-\lambda}}\no{\gamma_{\lambda''}(w)\beta_{\lambda'''}(w)\ee^{-c}(w)}\right) \\
        & - \frac{1}{2}\sum_{\substack{\alpha',\alpha'' \in \proots \setminus \hroot \\ \alpha'+\alpha''=\hroot+\alpha}} \no{\beta_{\alpha'} (z) \beta_{\alpha''}(z)\ee^{-c}(z)} \left(\gamma_\lambda(w) 
        + \sum_{\substack{\lambda'',\lambda''' \in \proots \setminus \hroot \\ -\lambda''+\lambda'''=\hroot-\lambda}}\no{\gamma_{\lambda''}(w)\beta_{\lambda'''}(w)\ee^{-c}(w)}\right). \notag
    \end{align}
    If $\beta_\alpha(z) \no{\gamma_{\lambda''}(w)\beta_{\lambda'''}(w)\ee^{-c}(w)}$ is singular, it must be the case that $\lambda'' = \alpha$. So $-\lambda''+\lambda'''= -\alpha+\lambda'''=\hroot-\lambda$. Rearranging gives $\lambda+\lambda'''=\hroot+\alpha$. Hence there are two addition singular contributions from the $\alpha',\alpha''$-sum's operator product expansion with $\gamma_\lambda(w)$, namely those corresponding to $\alpha'=\lambda, \alpha''=\lambda'''$ and  $\alpha'=\lambda''', \alpha''=\lambda$. These three contributions cancel, so we are left with
    \begin{equation} \label{tildebtildegInt}
        \tilde{\beta}_\alpha(z) \tilde{\gamma}_\lambda(w) \sim \beta_\alpha(z) \gamma_\lambda(w)  - \frac{1}{2}\sum_{\substack{\alpha',\alpha'', \lambda'',\lambda'''  \in \proots \setminus \hroot \\ \alpha'+\alpha''=\hroot+\alpha \\ -\lambda''+\lambda'''=\hroot-\lambda}} \no{\beta_{\alpha'} (z) \beta_{\alpha''}(z)\ee^{-c}(z)} \no{\gamma_{\lambda''}(w)\beta_{\lambda'''}(w)\ee^{-c}(w)}.
    \end{equation}
    
    If the summand in \eqref{tildebtildegInt} is nonsingular, then either $\alpha'=\lambda''$ or $\alpha''=\lambda''$. This cannot be the case as, by \cref{lem:InternalExposed}, $\alpha'$ and $\alpha''$ are exposed and $\lambda''$ is internal. Finally all that remains is to compute $\tilde{\gamma}_\alpha(z) \tilde{\gamma}_\lambda(w)$:
    \begin{align}
        \tilde{\gamma}_\alpha(z) \tilde{\gamma}_\lambda(w) \sim& \ 
    \left(\gamma_\alpha(z) 
    + \sum_{\substack{\alpha' \in \proots \setminus \hroot \\ \alpha'=\hroot-\alpha}}\no{\beta_{\alpha'}(z)\ee^{-c}(z)}
    + \sum_{\substack{\alpha'',\alpha''' \in \proots \setminus \hroot \\ -\alpha''+\alpha'''=\hroot-\alpha}}\no{\gamma_{\alpha''}(z)\beta_{\alpha'''}(z)\ee^{-c}(z)}\right)\\
    &\cdot\left(\gamma_\lambda(w) 
    + \sum_{\substack{\lambda' \in \proots \setminus \hroot \\ \lambda'=\hroot-\lambda}}\no{\beta_{\lambda'}(w)\ee^{-c}(w)}
    + \sum_{\substack{\lambda'',\lambda''' \in \proots \setminus \hroot \\ -\lambda''+\lambda'''=\hroot-\lambda}}\no{\gamma_{\lambda''}(w)\beta_{\lambda'''}(w)\ee^{-c}(w)}\right). \notag
    \end{align}
    Expanding the brackets,
    \begin{align}
        \tilde{\gamma}_\alpha(z) \tilde{\gamma}_\lambda(w) \sim&\ 
        \sum_{\substack{\lambda' \in \proots \setminus \hroot \\ \lambda'=\hroot-\lambda}}\delta_{\alpha,\lambda'}\frac{\ee^{-c}(w)}{z-w} + 
        \sum_{\substack{\lambda'',\lambda''' \in \proots \setminus \hroot \\ -\lambda''+\lambda'''=\hroot-\lambda}}\delta_{\alpha,\lambda'''}\frac{\no{\gamma_{\lambda''}(w)\ee^{-c}(w)}}{z-w}\\ &-
        \sum_{\substack{\alpha' \in \proots \setminus \hroot \\ \alpha'=\hroot-\alpha}}\delta_{\alpha',\lambda}\frac{\ee^{-c}(w)}{z-w} - 
        \sum_{\substack{\alpha'',\alpha''' \in \proots \setminus \hroot \\ -\alpha''+\alpha'''=\hroot-\alpha}}\delta_{\alpha''',\lambda}\frac{\no{\gamma_{\alpha''}(w)\ee^{-c}(w)}}{z-w} \notag\\&+
        \left(\sum_{\substack{\alpha'',\alpha''' \in \proots \setminus \hroot \\ -\alpha''+\alpha'''=\hroot-\alpha}}\no{\gamma_{\alpha''}(z)\beta_{\alpha'''}(z)\ee^{-c}(z)}\right)
        \left(\sum_{\substack{\lambda'',\lambda''' \in \proots \setminus \hroot \\ -\lambda''+\lambda'''=\hroot-\lambda}}\no{\gamma_{\lambda''}(w)\beta_{\lambda'''}(w)\ee^{-c}(w)}\right). \notag
    \end{align}
    
    Everything except for the last line above cancels. To deal with it, note that $\alpha''$ and $\lambda''$ are internal and $\alpha'''$ and $\lambda'''$ are exposed by \cref{lem:InternalExposed}, so we therefore have $\tilde{\gamma}_\alpha(z) \tilde{\gamma}_\lambda(w) \sim 0$.
    \end{proof}
    
\begin{proposition} \label{prop:tildeHalfLatAffOPEs}
    The vertex subalgebras $\widetilde{\Pi}$ and $\widetilde{\uaffvoa{\kk+\dcox}{\cartan}}$ are isomorphic to $\Pi$ and $\uaffvoa{\kk+\dcox}{\cartan}$ respectively.
\end{proposition}
\begin{proof}
     All the affine fields $\sroot{i}(z)$ have nonsingular operator product expansion with $c(z)$ so the fields $\set{\tsroot{i}(z)}$ have the same operator product expansions as $\set{\sroot{i}(z)}$. Since $\uaffvoa{\kk+\dcox}{\cartan}$ is simple and universal, we therefore have $\uaffvoa{\kk+\dcox}{\cartan} \cong \widetilde{\uaffvoa{\kk+\dcox}{\cartan}}$. For $\widetilde{\Pi}$, consider the map $\widetilde{\Pi} \rightarrow \Pi$ defined by taking a generating field in \eqref{eq:tildeHalfLat} to its untilded version. The operator product expansions satisfied by $\ee^{m\tilde{c}}(z)$ are clearly respected by this map, as are those satisfied by $\tilde{c}(z)$. What remains is to compute $\tilde{d}(z)\tilde{d}(w)$ and see that the map sends the result to the operator product $d(z)d(w) \sim 0$. In simpler terms, we must check that $\tilde{d}(z)\tilde{d}(w) \sim 0$. Now,
    \begin{align}
        \tilde{d}(z)\tilde{d}(w) \sim& \ - (\kk+\dcox)\inner{\fwt{n}}{\fwt{n}}d(z) c(w) - (\kk+\dcox)\inner{\fwt{n}}{\fwt{n}}c(z)d(w) + 4 \fwt{n}(z)\fwt{n}(w) \\
        &\hspace{2em} - d(z) \left( \sum_{\substack{\lambda,\lambda' \in \proots \setminus \hroot \\ \lambda+\lambda'=\hroot}} \no{\beta_\lambda (w) \beta_{\lambda'}(w)\ee^{-c}(w)} 
        +\sum_{\substack{\lambda,\lambda',\lambda'' \in \proots \setminus \hroot \\ -\lambda+\lambda'+\lambda''=\hroot}} \no{\gamma_\lambda(w) \beta_{\lambda'} (w) \beta_{\lambda''}(w)\ee^{-c}(w)}\right) \notag\\
        &\hspace{4em} -\left(\sum_{\substack{\alpha,\alpha' \in \proots \setminus \hroot \\ \alpha+\alpha'=\hroot}} \no{\beta_\alpha (z) \beta_{\alpha'}(z)\ee^{-c}(z)} 
        + \sum_{\substack{\alpha,\alpha',\alpha'' \in \proots \setminus \hroot \\ -\alpha+\alpha'+\alpha''=\hroot}} \no{\gamma_\alpha(z) \beta_{\alpha'} (z) \beta_{\alpha''}(z)\ee^{-c}(z)}\right)d(w) \notag\\
        &\hspace{6em} + \left( \sum_{\substack{\alpha,\alpha' \in \proots \setminus \hroot \\ \alpha+\alpha'=\hroot}} \no{\beta_\alpha (z) \beta_{\alpha'}(z)\ee^{-c}(z)} 
        + \sum_{\substack{\alpha,\alpha',\alpha'' \in \proots \setminus \hroot \\ -\alpha+\alpha'+\alpha''=\hroot}} \no{\gamma_\alpha(z) \beta_{\alpha'} (z) \beta_{\alpha''}(z)\ee^{-c}(z)}\right) \notag\\
        &\hspace{8em} \cdot\left( \sum_{\substack{\lambda,\lambda' \in \proots \setminus \hroot \\ \lambda+\lambda'=\hroot}} \no{\beta_\lambda (w) \beta_{\lambda'}(w)\ee^{-c}(w)} 
        +\sum_{\substack{\lambda,\lambda',\lambda'' \in \proots \setminus \hroot \\ -\lambda+\lambda'+\lambda''=\hroot}} \no{\gamma_\lambda(w) \beta_{\lambda'} (w) \beta_{\lambda''}(w)\ee^{-c}(w)}\right). \notag
    \end{align}
    
    The first three terms on the right hand side sum to zero, while those involving $d(z)$ and $d(w)$ cancel since $d(z) \ee^{-c}(w) \sim -\ee^{-c}(z)d(w)$. Therefore
    \begin{align} \label{eq:ddOPE}
        \tilde{d}(z)\tilde{d}(w) \sim& \sum_{\substack{\alpha,\alpha' \in \proots \setminus \hroot \\ \alpha+\alpha'=\hroot}} \no{\beta_\alpha (z) \beta_{\alpha'}(z)\ee^{-c}(z)} 
        \left(\sum_{\substack{\lambda,\lambda',\lambda'' \in \proots \setminus \hroot \\ -\lambda+\lambda'+\lambda''=\hroot}} \no{\gamma_\lambda(w) \beta_{\lambda'} (w) \beta_{\lambda''}(w)\ee^{-c}(w)}\right) \notag \\
        &\hspace{-4em}+  \sum_{\substack{\alpha,\alpha',\alpha'' \in \proots \setminus \hroot \\ -\alpha+\alpha'+\alpha''=\hroot}} \no{\gamma_\alpha(z) \beta_{\alpha'} (z) \beta_{\alpha''}(z)\ee^{-c}(z)} 
        \left( \sum_{\substack{\lambda,\lambda' \in \proots \setminus \hroot \\ \lambda+\lambda'=\hroot}} \no{\beta_\lambda (w) \beta_{\lambda'}(w)\ee^{-c}(w)} 
        +\sum_{\substack{\lambda,\lambda',\lambda'' \in \proots \setminus \hroot \\ -\lambda+\lambda'+\lambda''=\hroot}} \no{\gamma_\lambda(w) \beta_{\lambda'} (w) \beta_{\lambda''}(w)\ee^{-c}(w)}\right) \notag \\
        \sim& \left( \sum_{\substack{\alpha,\alpha',\alpha'' \in \proots \setminus \hroot \\ -\alpha+\alpha'+\alpha''=\hroot}} \no{\gamma_\alpha(z) \beta_{\alpha'} (z) \beta_{\alpha''}(z)\ee^{-c}(z)} \right) 
        \left(\sum_{\substack{\lambda,\lambda',\lambda'' \in \proots \setminus \hroot \\ -\lambda+\lambda'+\lambda''=\hroot}} \no{\gamma_\lambda(w) \beta_{\lambda'} (w) \beta_{\lambda''}(w)\ee^{-c}(w)}\right). \notag 
    \end{align}
    But the right hand side is zero since, by \cref{lem:InternalExposed}, the roots appearing in the $\gamma_\alpha(z)$'s are internal while those appearing in the $\beta_\alpha(z)$'s are exposed. Hence $\tilde{d}(z)\tilde{d}(w) \sim 0$ as required. Therefore, as $\Pi$ and $\widetilde{\Pi}$ are simple and universal, $\widetilde{\Pi} \cong \Pi$.
\end{proof}

\noindent By \cref{prop:tildeGhostOPEs} and \cref{prop:tildeHalfLatAffOPEs}, we have a number of subalgebras of $\uaffvoa{\kk+\dcox}{\cartan} \otimes \Pi \otimes \bigotimes_{\alpha \in \proots \setminus \hroot} \bgvoa_\alpha$, all of which are isomorphic to their untilded versions. It is then natural to ask whether
\begin{equation} \label{eq:FreeFieldIso}
    \uaffvoa{\kk+\dcox}{\cartan} \otimes \Pi  \otimes \bigotimes_{\alpha \in \proots \setminus \hroot} \bgvoa_\alpha= 
    \widetilde{\uaffvoa{\kk+\dcox}{\cartan}} \otimes \widetilde{\Pi} \otimes \bigotimes_{\alpha \in \proots \setminus \hroot} \widetilde{\bgvoa_\alpha}.
\end{equation}
This is the case if the tilded subalgebras are all mutually-orthogonal (in the sense that operator product expansions of fields in two different tilded subalgebras are nonsingular) and that the equations defining the tilded subalgebras are invertible.

\begin{proposition} \label{prop:FreeFieldIso}
Let $\kk$ be noncritical. Then,
    \begin{equation}
    \uaffvoa{\kk+\dcox}{\cartan} \otimes \Pi  \otimes \bigotimes_{\alpha \in \proots \setminus \hroot} \bgvoa_\alpha= 
    \widetilde{\uaffvoa{\kk+\dcox}{\cartan}} \otimes \widetilde{\Pi} \otimes \bigotimes_{\alpha \in \proots \setminus \hroot} \widetilde{\bgvoa_\alpha}.
\end{equation}
\end{proposition}
\begin{proof}
\cref{prop:tildeGhostOPEs} implies that the subalgebras $\set{\widetilde{\bgvoa_\alpha}}_{\alpha \in \proots \setminus \hroot}$ are mutually orthogonal. It is clear that $ \bigotimes_{\alpha \in \proots \setminus \hroot} \widetilde{\bgvoa_\alpha}$ and $\widetilde{\uaffvoa{\kk+\dcox}{\cartan}}$ are orthogonal, while $\widetilde{\Pi}$ and $\widetilde{\uaffvoa{\kk+\dcox}{\cartan}}$ are orthogonal since
\begin{align}
    \tilde{d}(z)\tsroot{i}(w) &\sim (d(z) + 2 \fwt{n}(z))(\sroot{i}(w)-\delta_{i,n}(\kk+\dcox)c(w)) \\
    &\sim \frac{-2\delta_{i,n}(\kk+\dcox)\wun(w)}{(z-w)^2} + \frac{2(\kk+\dcox)\inner{\fwt{n}}{\sroot{i}}}{(z-w)^2} \notag \\
    &\sim 0.\notag 
\end{align}
To see that $\widetilde{\Pi}$ and $ \bigotimes_{\alpha \in \proots \setminus \hroot} \widetilde{\bgvoa_\alpha}$ are orthogonal, it suffices to check that $\tilde{d}(z) \tilde{\beta}_\alpha(w) \sim 0 \sim \tilde{d}(z) \tilde{\gamma}_\alpha(w)$ for all $\alpha \in \proots \setminus \hroot$. For this, observe that the $\beta_\alpha(z)$'s that appear in the expression for $\tilde{d}(z)$ in \eqref{eq:tildeHalfLat} all correspond to exposed $\alpha$'s, while the $\gamma_\alpha(z)$'s correspond to internal $\alpha$'s. So
\begin{align}
    \tilde{d}(z) \tilde{\beta}_\alpha(w) &\sim
    - \frac{1}{2}\sum_{\substack{\alpha',\alpha'' \in \proots \setminus \hroot \\ \alpha'+\alpha''=\hroot+\alpha}} d(z)\no{\beta_{\alpha'} (w) \beta_{\alpha''}(w)\ee^{-c}(w)}
    - \sum_{\substack{\lambda,\lambda',\lambda'' \in \proots \setminus \hroot \\ -\lambda+\lambda'+\lambda''=\hroot}} \no{\gamma_\lambda(z) \beta_{\lambda'} (z) \beta_{\lambda''}(z)\ee^{-c}(z)} \beta_\alpha(w), \\ &\sim
    \sum_{\substack{\alpha',\alpha'' \in \proots \setminus \hroot \\ \alpha'+\alpha''=\hroot+\alpha}} \no{\beta_{\alpha'} (w) \beta_{\alpha''}(w)\ee^{-c}(w)}(z-w)^{-1} 
     - \sum_{\substack{\lambda',\lambda'' \in \proots \setminus \hroot \\ -\alpha+\lambda'+\lambda''=\hroot}} \no{\beta_{\lambda'} (z) \beta_{\lambda''}(z)\ee^{-c}(z)}(z-w)^{-1} \notag\\&\sim 0. \notag
\end{align}
Similarly, 
\begin{align}
    \tilde{d}(z) \tilde{\gamma}_\alpha(w) &\sim d(z) \left(\sum_{\substack{\alpha' \in \proots \setminus \hroot \\ \alpha'=\hroot-\alpha}}\no{\beta_{\alpha'}(z)\ee^{-c}(z)}
    + \sum_{\substack{\alpha'',\alpha''' \in \proots \setminus \hroot \\ -\alpha''+\alpha'''=\hroot-\alpha}}\no{\gamma_{\alpha''}(z)\beta_{\alpha'''}(z)\ee^{-c}(z)} \right) \\
    &\hspace{2em} -\left( \sum_{\substack{\lambda,\lambda' \in \proots \setminus \hroot \\ \lambda+\lambda'=\hroot}} \no{\beta_\lambda (z) \beta_{\lambda'}(z)\ee^{-c}(z)} 
    +\sum_{\substack{\lambda,\lambda',\lambda'' \in \proots \setminus \hroot \\ -\lambda+\lambda'+\lambda''=\hroot}} \no{\gamma_\lambda(z) \beta_{\lambda'} (z) \beta_{\lambda''}(z)\ee^{-c}(z)}\right) \gamma_\alpha(w) \notag\\
    &\sim 2 \left(\sum_{\substack{\alpha' \in \proots \setminus \hroot \\ \alpha'=\hroot-\alpha}}\no{\beta_{\alpha'}(z)\ee^{-c}(z)}
    + \sum_{\substack{\alpha'',\alpha''' \in \proots \setminus \hroot \\ -\alpha''+\alpha'''=\hroot-\alpha}}\no{\gamma_{\alpha''}(z)\beta_{\alpha'''}(z)\ee^{-c}(z)} \right)(z-w)^{-1} \notag\\
    &\hspace{2em} -\left( \sum_{\substack{\lambda' \in \proots \setminus \hroot \\ \alpha+\lambda'=\hroot}} \no{ \beta_{\lambda'}(z)\ee^{-c}(z)} 
    +\sum_{\substack{\lambda,\lambda'' \in \proots \setminus \hroot \\ -\lambda+\alpha+\lambda''=\hroot}} \no{\gamma_\lambda(z)  \beta_{\lambda''}(z)\ee^{-c}(z)}\right) (z-w)^{-1}
    \notag\\
    &\hspace{3em} -\left( \sum_{\substack{\lambda \in \proots \setminus \hroot \\ \lambda+\alpha=\hroot}} \no{\beta_\lambda (z) (z)\ee^{-c}(z)} 
    +\sum_{\substack{\lambda,\lambda' \in \proots \setminus \hroot \\ -\lambda+\lambda'+\alpha=\hroot}} \no{\gamma_\lambda(z) \beta_{\lambda'} (z) \ee^{-c}(z)}\right) (z-w)^{-1} \notag\\
    &\sim 0. \notag
\end{align}
So all the tilded subalgebras are mutually orthogonal. 

For invertibility, expressing $\sroot{i}(z)$, $c(z)$ and $\ee^{mc}(z)$ in terms of tilded fields is straightforward. For $\beta_\alpha(z)$ and $\gamma_\alpha(z)$, recall from \cref{lem:InternalExposed} that the ghost fields appearing in the summands of \eqref{eq:tildeGhosts} are all equal to their tilded versions. So, we can write
\begin{equation} \label{eq:untildedB}
    \beta_\alpha(z) = 
    \tilde{\beta}_\alpha(z) + \frac{1}{2}\sum_{\substack{\alpha',\alpha'' \in \proots \setminus \hroot \\ \alpha'+\alpha''=\hroot+\alpha}} \no{\beta_{\alpha'} (z) \beta_{\alpha''}(z)\ee^{-c}(z)} =
    \tilde{\beta}_\alpha(z) + \frac{1}{2}\sum_{\substack{\alpha',\alpha'' \in \proots \setminus \hroot \\ \alpha'+\alpha''=\hroot+\alpha}} \no{\tilde{\beta}_{\alpha'} (z) \tilde{\beta}_{\alpha''}(z)\ee^{-c}(z)}
\end{equation}
and likewise for $\gamma_\alpha(z)$ and $d(z)$.
\end{proof}

The question of which subalgebra of $\uaffvoa{\kk+\dcox}{\cartan} \otimes \Pi \otimes \bigotimes_{\alpha \in \proots \setminus \hroot} \bgvoa_\alpha$ the first $n$ screening operators in \eqref{eq:MinToAffIntermediate} act on becomes significantly easier by reframing everything in terms of tilded fields. The key question is then, by \cref{prop:FreeFieldIso}, which subalgebra of $\widetilde{\uaffvoa{\kk+\dcox}{\cartan}} \otimes \widetilde{\Pi} \otimes \bigotimes_{\alpha \in \proots \setminus \hroot} \widetilde{\bgvoa_\alpha}$ the first $n$ screening operators in \eqref{eq:MinToAffIntermediate} act on. In order to answer this, we have the following convenient result:

\begin{proposition}
    For all $i \in \set{1,\dots,n-1}$, 
    \begin{align}
        S_i(z) =& \no{\left( \beta_i(z) + \sum_{j=1}^{i-1} \gamma_{i-j,i-1}(z) \beta_{i-j,i}(z)  \right) \ee^{\frac{-1}{\kk+n+1} \alpha_i}(z) } \\
        =& \no{\left( \tilde{\beta}_i(z) + \sum_{j=1}^{i-1} \tilde{\gamma}_{i-j,i-1}(z) \tilde{\beta}_{i-j,i}(z)  \right) \ee^{\frac{-1}{\kk+n+1} \tilde{\alpha}_i}(z) } = \tilde{S}_i(z). \notag
    \end{align}
\end{proposition}
\begin{proof}
    Consider first the $i=1$ case. Then $S_1(z)=\no{\beta_1(z)\ee^{\frac{-1}{\kk+n+1} \alpha_1}(z) } = \no{\tilde{\beta}_1(z)\ee^{\frac{-1}{\kk+n+1} \tilde{\alpha}_1}(z) = \tilde{S}_1(z)}$.
    Now let $1<i<n-1$. Then $\sroot{i}$ is an internal root and, using \eqref{eq:untildedB},
    \begin{align}
        S_i(z) &= \no{\left( \beta_i(z) + \sum_{j=1}^{i-1} \gamma_{i-j,i-1}(z) \beta_{i-j,i}(z)  \right) \ee^{\frac{-1}{\kk+n+1} \alpha_i}(z) } \\
        &= \no{\left( \tilde{\beta}_i(z)+\no{\tilde{\beta}_{1,i} (z) \tilde{\beta}_{i,n}(z)\ee^{-c}(z)} + \sum_{j=1}^{i-1} \gamma_{i-j,i-1}(z) \beta_{i-j,i}(z)  \right) \ee^{\frac{-1}{\kk+n+1} \alpha_i}(z) }. \notag
    \end{align}
    The $j=i-1$ term in the sum is
    \begin{equation}
        \no{\gamma_{1,i-1}(z) \beta_{1,i}(z)} = \no{\left(\tilde{\gamma}_{1,i-1}(z)-\tilde{\beta}_{i,n}(z)\ee^{-c}(z)
        - \sum_{k=2}^{i-1} \tilde{\gamma}_{k,i-1}(z)\tilde{\beta}_{k,n}(z)\ee^{-c}(z)\right) \tilde{\beta}_{1,i}(z)},
    \end{equation}
    where the sum is empty when $i=2$. In that case, it is clear that $S_2(z)=\tilde{S}_2(z)$. For $i>2$, the $j<i-1$ term in the sum is
    \begin{equation}
        \no{\gamma_{i-j,i-1}(z) \beta_{i-j,i}(z)} = \no{\tilde{\gamma}_{i-j,i-1}(z) \left( \tilde{\beta}_{i-j,i}(z)+ \tilde{\beta}_{1,i} (z) \tilde{\beta}_{i-j,n}(z)\ee^{-c}(z)\right)}.
    \end{equation}
    
    Adding all the contributions together and using that $\tilde{\alpha}_i(z) = \alpha_i(z)$ for $i<n$, we obtain
    \begin{align}
        S_i(z) =&\ \no{\left( \tilde{\beta}_i(z)+\no{\tilde{\beta}_{1,i} (z) \tilde{\beta}_{i,n}(z)\ee^{-c}(z)}\right) \ee^{\frac{-1}{\kk+n+1} \alpha_i}(z) }\\
        &+  \no{\left(\tilde{\gamma}_{1,i-1}(z)\tilde{\beta}_{1,i}(z) -\tilde{\beta}_{1,i}(z) \tilde{\beta}_{i,n}(z)\ee^{-c}(z)
        - \sum_{k=2}^{i-1} \tilde{\gamma}_{k,i-1}(z)\tilde{\beta}_{1,i}(z) \tilde{\beta}_{k,n}(z)\ee^{-c}(z) \right) \ee^{\frac{-1}{\kk+n+1} \alpha_i}(z) }\notag\\
        &+ \sum_{j=1}^{i-2} \no{ \left( \tilde{\gamma}_{i-j,i-1}(z) \tilde{\beta}_{i-j,i}(z)+ \tilde{\gamma}_{i-j,i-1}(z)\tilde{\beta}_{1,i} (z) \tilde{\beta}_{i-j,n}(z)\ee^{-c}(z) \right) \ee^{\frac{-1}{\kk+n+1} \alpha_i}(z) }\notag\\
        =&\ \no{\left(\tilde{\beta}_i(z)+\tilde{\gamma}_{1,i-1}(z)\tilde{\beta}_{1,i}(z)+ \sum_{j=1}^{i-2} \tilde{\gamma}_{i-j,i-1}(z) \tilde{\beta}_{i-j,i}(z)\right) \ee^{\frac{-1}{\kk+n+1} \alpha_i}(z) }\notag\\
        &-  \no{\left( \sum_{k=2}^{i-1} \tilde{\gamma}_{k,i-1}(z)\tilde{\beta}_{1,i}(z) \tilde{\beta}_{k,n}(z)\ee^{-c}(z) \right) \ee^{\frac{-1}{\kk+n+1} \alpha_i}(z) }
        + \sum_{j=1}^{i-2} \no{ \left(  \tilde{\gamma}_{i-j,i-1}(z)\tilde{\beta}_{1,i} (z) \tilde{\beta}_{i-j,n}(z)\ee^{-c}(z) \right) \ee^{\frac{-1}{\kk+n+1} \alpha_i}(z) }\notag\\
        =&\ \no{\left(\tilde{\beta}_i(z)+ \sum_{j=1}^{i-1} \tilde{\gamma}_{i-j,i-1}(z) \tilde{\beta}_{i-j,i}(z)\right) \ee^{\frac{-1}{\kk+n+1} \tilde{\alpha}_i}(z) }\notag\\
        =&\ \tilde{S}_i(z). \notag
    \end{align}
    for $i>2$.
\end{proof}

\begin{corollary}
The screening operators $\int S_i(z) \dd z =\int \tilde{S}_i(z) \dd z $ for $i<n$ and $\int \no{\tilde{\gamma}_{1,n-1}(z) \ee^{\frac{-1}{\kk+\dcox} \tsroot{n}}(z)} \dd z $ act as zero on the subalgebra
\begin{equation}
    \widetilde{\Pi} \otimes \bigotimes_{i=2}^{n} \widetilde{\bgvoa_{\sroot{i,n}}} \subset \uaffvoa{\kk+\dcox}{\cartan} \otimes \Pi \otimes \bigotimes_{\alpha \in \proots \setminus \hroot} \bgvoa_\alpha = \left( \widetilde{\uaffvoa{\kk+\dcox}{\cartan}} \otimes \bigotimes_{j=1}^{n-1}\bigotimes_{i=1}^{j} \widetilde{\bgvoa_{\sroot{i,j}}} \right) \otimes \left( \widetilde{\Pi} \otimes \bigotimes_{i=2}^{n} \widetilde{\bgvoa_{\sroot{i,n}}} \right).
\end{equation}
\end{corollary}

To summarise what we have found thus far, the Wakimoto realisation of $\uaffvoa{\kk}{\slnpone}$ is an embedding of $\uaffvoa{\kk}{\slnpone}$ into $\uaffvoa{\kk+\dcox}{\cartan} \otimes \bigotimes_{\alpha \in \proots} \bgvoa_\alpha$. Composing with the bosonisation of $\bgvoa_\theta$ defines an embedding into $\uaffvoa{\kk+\dcox}{\cartan} \otimes \Pi \otimes \bigotimes_{\alpha \in \proots \setminus \theta} \bgvoa_\alpha$ which by \cref{prop:FreeFieldIso} is equal to $\widetilde{\uaffvoa{\kk+\dcox}{\cartan}} \otimes \widetilde{\Pi} \otimes \bigotimes_{\alpha \in \proots \setminus \hroot} \widetilde{\bgvoa_\alpha}$. 

For generic $\kk$, the embedding $\uaffvoa{\kk}{\slnpone} \hookrightarrow \widetilde{\uaffvoa{\kk+\dcox}{\cartan}} \otimes \widetilde{\Pi} \otimes \bigotimes_{\alpha \in \proots \setminus \hroot} \widetilde{\bgvoa_\alpha}$ is described by screening operators according to
\begin{equation}
     \uaffvoa{\kk}{\slnpone} \cong \left( \bigcap_{i=1}^{n-1} \textup{ker} \int \tilde{S}_i(z) \ \dd z \right) \cap \textup{ker} \int \no{\tilde{\gamma}_{1,n-1}(z) \ee^{\frac{-1}{\kk+\dcox} \tsroot{n}}(z)} \ \dd z \cap  \textup{ker} \int \ee^{\frac{1}{2}c + \frac{1}{2}d}(z) \ dz,
\end{equation}
where the first $n$ screening operators only act nontrivially on $\widetilde{\uaffvoa{\kk+\dcox}{\cartan}} \otimes\bigotimes_{j=1}^{n-1}\bigotimes_{i=1}^{j} \widetilde{\bgvoa_{\sroot{i,j}}}$. This describes precisely the Wakimoto realisation of $\uQHR{\kk}{\slnpone}{\fmin}$ as per \eqref{eq:WakAffandMin} using the isomorphism between $\widetilde{\uaffvoa{\kk+\dcox}{\cartan}} \otimes \bigotimes_{j=1}^{n-1}\bigotimes_{i=1}^{j} \widetilde{\bgvoa_{\sroot{i,j}}}$ and $\uaffvoa{\kk+\dcox}{\cartan} \otimes \bigotimes_{j=1}^{n-1}\bigotimes_{i=1}^{j} \bgvoa_{\sroot{i,j}}$ defined by removing the tildes on all fields. Thus the fields in $\widetilde{\uaffvoa{\kk+\dcox}{\cartan}} \otimes \bigotimes_{j=1}^{n-1}\bigotimes_{i=1}^{j} \widetilde{\bgvoa_{\sroot{i,j}}}$ that appear in the Wakimoto realisation of $\uaffvoa{\kk}{\slnpone}$ can be identified with fields in $\uQHR{\kk}{\slnpone}{\fmin}$.

\begin{theorem} \label{thm:mintoaff}
Let $\kk$ be generic. There exists an embedding 
\begin{equation}
    \uaffvoa{\kk}{\slnpone} \hookrightarrow \uQHR{\kk}{\slnpone}{\fmin} \otimes \Pi \otimes \bigotimes_{i=2}^{n} \bgvoa_{\sroot{i,n}}
\end{equation}
whose image is specified by
\begin{equation}
    \uaffvoa{\kk}{\slnpone} \cong \textup{\normalfont ker} \int  \ee^{A}(z) \ \dd z,
\end{equation}
where
\begin{align} \label{eq:finalscreening}
     A(z) = &\ \frac{1}{2}\left(1-(\kk+\dcox)\inner{\fwt{n}}{\fwt{n}} \right) c(z) + \frac{1}{2}d(z)  - \omega_n(z) \\
    &+ \frac{1}{2} \sum_{\substack{\alpha,\alpha' \in \proots \setminus \hroot \\ \alpha+\alpha'=\hroot}} \no{\beta_\alpha (z) \beta_{\alpha'}(z)\ee^{-c}(z)} 
    +\frac{1}{2} \sum_{\substack{\alpha,\alpha',\alpha'' \in \proots \setminus \hroot \\ -\alpha+\alpha'+\alpha''=\hroot}} \no{\gamma_\alpha(z) \beta_{\alpha'} (z) \beta_{\alpha''}(z)\ee^{-c}(z)}. \notag
\end{align}
Here, the fields $\omega_n(z)$, $\beta_\alpha(z)$ and $\gamma_\alpha(z)$, for $\alpha =\sroot{i,j}$ with $ j \neq n$, act on $\uQHR{\kk}{\slnpone}{\fmin}$ via the Wakimoto realisation of $\uQHR{\kk}{\slnpone}{\fmin}$.
\end{theorem}
\begin{proof}
The proof here is essentially the same as that of \cite[Thm.~3.3]{Feh21c}. We reproduce it here for completeness. 

Using composition of the Wakimoto realisation and FMS bosonisation of $\bgvoa_{\sroot{1,n}}$, and rewriting the result in terms of tilded fields, we can decompose the fields of $\uaffvoa{\kk}{\slnpone}$ as 
\begin{equation}
    F(z) = \sum_m \tilde{A}_m(z) \otimes \tilde{B}_m(z), \qquad 
    \tilde{A}_m(z) \in \widetilde{\uaffvoa{\kk+\dcox}{\cartan}} \otimes \bigotimes_{j=1}^{n-1}\bigotimes_{i=1}^{j} \widetilde{\bgvoa_{\sroot{i,j}}} \ \text{  and  } \
    \tilde{B}_m(z) \in  \widetilde{\Pi} \otimes  \bigotimes_{i=2}^{n} \widetilde{\bgvoa_{\sroot{i,n}}}.
\end{equation}
We may assume that the fields $\tilde{B}_m(z)$ are linearly independent for if they are not, then we can simply redefine the fields $\tilde{A}_m(z)$ and reduce the range of $m$ such that this is the case. 

The screening operators
\begin{equation}
    \int \tilde{S}_i(z) \ \dd z, \qquad \int \no{\tilde{\gamma}_{1,n-1}(z) \ee^{\frac{-1}{\kk+\dcox} \tsroot{n}}(z)} \ \dd z
\end{equation}
for $i =1, \dots, n-1$ act trivially on $\tilde{B}_m(w)$. So if $F(z)$ is in the intersection of their kernels, we must have that 

\begin{equation} \label{eq:findingreg}
   \int \tilde{S}_i(z) \tilde{A}_m(w)\ \dd z=0, \qquad \int \no{\tilde{\gamma}_{1,n-1}(z) \ee^{\frac{-1}{\kk+\dcox} \tsroot{n}}(z)} \tilde{A}_m(w) \ \dd z=0
\end{equation}
for all $i = 1,\dots, n-1$ and $m$ by the linear independence of the fields $\tilde{B}_m(z)$. Therefore, since $F(z) \in \uaffvoa{\kk}{\slnpone}$,
\begin{equation}
     \tilde{A}_m(z) \in \left( \bigcap_{i=1}^{n-1} \textup{ ker} \int \tilde{S}_i(z) \ \dd z \right) \cap \textup{ker} \int \no{\tilde{\gamma}_{1,n-1}(z) \ee^{\frac{-1}{\kk+\dcox} \tsroot{n}}(z)} \ \dd z  \quad \text{for all } m.
\end{equation}

The right hand side is equal to the image of the composition of the Wakimoto realisation of $\uQHR{\kk}{\slnpone}{\fmin}$ from \eqref{eq:WakAffandMin} with the isomorphisms $ \widetilde{\uaffvoa{\kk+\dcox}{\cartan}} \cong  \uaffvoa{\kk+\dcox}{\cartan}$ and $\widetilde{\bgvoa_{\sroot{i,j}}} \cong \bgvoa_{\sroot{i,j}}$. Hence we may treat the fields $\tilde{A}_m(z)$ as fields in $\uQHR{\kk}{\slnpone}{\fmin}$. 

The screening operator exponent \eqref{eq:finalscreening} is obtained by rewriting the FMS bosonisation screening field $\ee^{\frac{1}{2}c + \frac{1}{2}d}(z)$ using the definitions of the tilded fields. That this defines an embedding $\uaffvoa{\kk}{\slnpone} \hookrightarrow \uQHR{\kk}{\slnpone}{\fmin}\otimes \Pi \otimes \bigotimes_{i=2}^{n} \bgvoa_{\sroot{i,n}} $ follows from the fact that it is obtained by writing the embedded image of $\uaffvoa{\kk}{\slnpone}$ in terms of the embedded image of $\uQHR{\kk}{\slnpone}{\fmin}$ and other fields.
\end{proof}

\subsection{From hook-type to hook-type} \label{subsec:HooktoHook}
Recall from \cref{subsec:WakWalg} the two Wakimoto realisations of hook-type $\slnpone$ W-algebras corresponding to $\fhook{m}$ and $\bar{f}^{(m)}$ with screening operators 
\begin{equation} \label{eq:WakScreenersHook}
    \uQHR{\kk}{\slnpone}{f} \cong \bigcap_{i=1}^{m-1} \textup{ker} \int S_i(z) \ \dd z \cap 
    \left\{\begin{matrix}
        \textup{ker} \int \ee^{\frac{-1}{\kk+\dcox} \alpha_{m}}(z) \ \dd z, & f=\fhook{m} \\
        \textup{ker} \int \no{\gamma_{1,m-1}(z) \ee^{\frac{-1}{\kk+\dcox} \alpha_{m}}(z)} \ \dd z, & f=\bar{f}^{(m)}
    \end{matrix}\right\} \cap \bigcap_{i=m+1}^{n} \textup{ker} \int \ee^{\frac{-1}{\kk+\dcox} \alpha_{i}}(z) \ \dd z. 
\end{equation}
The number of bosonic ghost systems in the Wakimoto realisation of a W-algebra $ \uQHR{\kk}{\slnpone}{f}$ is determined by the cardinality of $(\proots)_0$. In the case of $\fhook{m}$ (and $\bar{f}^{(m)}$) and the good gradings specified earlier, the subset of the root lattice of $\slnpone$ generated by $(\sroots)_0$ is isomorphic to the root lattice for $\SLA{sl}{m}$ (empty for $m=1$). 

Even better, the screening fields $S_i(z)$ in the first $m-1$ screening operators in \eqref{eq:WakScreenersHook} are the same as those for the Wakimoto realisation of $\uaffvoa{\kk'}{\SLA{sl}{m}}$ but with a different domain and level $\kk'=\kk+n+1-m$. Likewise, the Wakimoto realisation of $\uQHR{\kk}{\slnpone}{\bar{f}^{(m-1)}}$ has the same ghost field content as the Wakimoto realisation of $\uQHR{\kk'}{\SLA{sl}{m}}{f_\theta}$ and several screening operators of the same form as those for $\uQHR{\kk'}{\SLA{sl}{m}}{f_\theta}$ but with a different domain and level. 

One might suspect from this that the $(m)$-to-$(m-1)$ hook-type inverse reduction is in many ways `the same' as the minimal-to-affine inverse reduction for $\SLA{sl}{m}$. This suspicion turns out to be correct, in that the results of the previous section can be easily adapted to prove the existence of the $(m)$-to-$(m-1)$ hook-type inverse reduction by identifying $(\proots)_0$ with the positive roots of $\SLA{sl}{m}$, the highest root $\hroot_0$ in $(\proots)_0$ with the highest root $\hroot$ of $\SLA{sl}{m}$ and appropriately shifting the Heisenberg fields.

To begin, consider again the Wakimoto realisation $\uQHR{\kk}{\slnpone}{\fhook{m}} \hookrightarrow \uaffvoa{\kk+\dcox}{\cartan} \otimes \bigotimes_{j=1}^{m-1}\bigotimes_{i=1}^{j} \bgvoa_{\sroot{i,j}}$, where we have observed that 
\begin{equation}
    (\proots)_0 = \set{\sroot{i,j} \ \vert \ 1 \leq i \leq j \leq m-1}.
\end{equation}
The highest root in $(\proots)_0$ is $\hroot_0=\sroot{1,m-1}$. Let $m >2$, which is when $(\proots)_0 \setminus \hroot_0$ is nonempty. The $m=2$ case is described in \cite{Feh21c} and the argument is essentially the same as what we are about to describe. Following the minimal-to-affine procedure, bosonise the bosonic ghost system $\bgvoa_{\hroot_0}$. The result is an embedding $\uQHR{\kk}{\slnpone}{\fhook{m}} \hookrightarrow \uaffvoa{\kk+\dcox}{\cartan} \otimes \Pi \otimes \bigotimes_{\alpha \in (\proots)_0 \setminus \hroot_0} \bgvoa_\alpha$ specified by
\begin{align} \label{eq:HookToHookIntermediate}
     \uQHR{\kk}{\slnpone}{\fhook{m}} \cong&\ \left( \bigcap_{i=1}^{m-2} \textup{ker} \int S_i(z) \ \dd z \right)\ \cap \ \textup{ker} \int \no{\tilde{\gamma}_{1,m-2}(z) \ee^{\frac{-1}{\kk+\dcox} \tsroot{m-1}}(z)} \ \dd z\ \\
     &\cap \bigcap_{i=m}^{n} \textup{ker} \int \ee^{\frac{-1}{\kk+\dcox} \alpha_{i}}(z) \ \dd z\ \cap\  \textup{ker} \int \ee^{\frac{1}{2}c + \frac{1}{2}d}(z) \ dz, \notag
\end{align}
where $\tsroot{m-1}(z) = \sroot{m-1}(z)-(\kk+\dcox)c(z)$ and 
\begin{equation}
    \tilde{\gamma}_{1,m-2}(z) = \gamma_{1,m-2}(z) + \no{ \beta_{m-1}(z)\ee^{-c}(z) + \sum_{j=1}^{m-3} \gamma_{m-j-1,m-2}(z) \beta_{m-j-1,m-1}(z)\ee^{-c}(z)}.
\end{equation}

Define $\tsroot{i}(z) = \sroot{i}(z)-\delta_{i,m-1}(\kk+\dcox)c(z)$ and let $\widetilde{\uaffvoa{\kk+\dcox}{\cartan}}$ be the subalgebra of $\uaffvoa{\kk+\dcox}{\cartan} \otimes \Pi \otimes \bigotimes_{\alpha \in (\proots)_0 \setminus \hroot_0} \bgvoa_\alpha$ generated by these fields. Likewise, let $\widetilde{\Pi}$ be the subalgebra generated by
\begin{equation} \label{eq:tildeHalfLatHook}
\begin{aligned}
    \tilde{c}(z)& = c(z), \qquad 
    \ee^{m\tilde{c}}(z) = \ee^{mc}(z), \\
    \tilde{d}(z) = d(z) 
    - (\kk+\dcox)\inner{\fwt{m-1}}{\fwt{m-1}} c(z)&
    + 2 \fwt{m-1}(z) \\
    - \sum_{\substack{\alpha,\alpha' \in  (\proots)_0 \setminus \hroot_0 \\ \alpha+\alpha'=\hroot_0}}& \no{\beta_\alpha (z) \beta_{\alpha'}(z)\ee^{-c}(z)} 
    - \sum_{\substack{\alpha,\alpha',\alpha'' \in  (\proots)_0 \setminus \hroot_0 \\ -\alpha+\alpha'+\alpha''=\hroot_0}} \no{\gamma_\alpha(z) \beta_{\alpha'} (z) \beta_{\alpha''}(z)\ee^{-c}(z)}.
\end{aligned}
\end{equation}
Finally for $\alpha \in  (\proots)_0 \setminus \hroot_0$, let $\widetilde{\bgvoa_\alpha}$ be the subalgebra of $\uaffvoa{\kk+\dcox}{\cartan} \otimes \Pi \otimes \bigotimes_{\alpha \in (\proots)_0 \setminus \hroot_0} \bgvoa_\alpha$ generated by the fields
\begin{equation} \label{eq:tildeGhostsHook}
\begin{aligned}
    \tilde{\beta}_\alpha(z) &= \beta_\alpha(z)
    - \frac{1}{2}\sum_{\substack{\alpha',\alpha'' \in  (\proots)_0 \setminus \hroot_0 \\ \alpha'+\alpha''=\hroot_0+\alpha}} \no{\beta_{\alpha'} (z) \beta_{\alpha''}(z)\ee^{-c}(z)}, \\
    \tilde{\gamma}_\alpha(z) &= \gamma_\alpha(z) 
    + \sum_{\substack{\alpha' \in (\proots)_0 \setminus \hroot_0 \\ \alpha'=\hroot_0-\alpha}}\no{\beta_{\alpha'}(z)\ee^{-c}(z)}
    + \sum_{\substack{\alpha'',\alpha''' \in (\proots)_0 \setminus \hroot_0 \\ -\alpha''+\alpha'''=\hroot_0-\alpha}}\no{\gamma_{\alpha''}(z)\beta_{\alpha'''}(z)\ee^{-c}(z)}.
\end{aligned}
\end{equation}

All the properties enjoyed by the tilded subalgebras in \cref{subsec:MintoAff} are satisfied by these subalgebras:
\begin{equation} \label{eq:hookffisos}
\begin{gathered}
    \widetilde{\uaffvoa{\kk+\dcox}{\cartan}} \cong \uaffvoa{\kk+\dcox}{\cartan}, \qquad \widetilde{\Pi} \cong \Pi, \qquad \widetilde{\bgvoa_\alpha} \cong \bgvoa_\alpha \quad \forall \alpha \in (\proots)_0 \setminus \hroot_0,\\
    \uaffvoa{\kk+\dcox}{\cartan} \otimes \Pi \otimes \bigotimes_{\alpha \in (\proots)_0 \setminus \hroot_0} \bgvoa_\alpha= \widetilde{\uaffvoa{\kk+\dcox}{\cartan}} \otimes \widetilde{\Pi} \otimes \bigotimes_{\alpha \in (\proots)_0 \setminus \hroot_0} \widetilde{\bgvoa_\alpha}.
\end{gathered}
\end{equation}
The proofs for all of these statements are identical to the corresponding statements in \cref{subsec:MintoAff} by replacing $n+1$ with $m$, $\proots$ with $(\proots)_0$ and $\hroot$ with $\hroot_0$, as well as modifying the definition of exposed (internal) roots to be  roots in $(\proots)_0$ (not) containing either $\sroot{1}$ or $\sroot{m}$ as a summand. Moreover, for $i<m-1$, we still have $S_i(z) = \tilde{S}_i(z)$. 

Hence, the Wakimoto realisation of $\uQHR{\kk}{\slnpone}{\fhook{m}}$ becomes 
\begin{align}
     \uQHR{\kk}{\slnpone}{\fhook{m}} \cong&\ \left( \bigcap_{i=1}^{m-2} \textup{ker} \int \tilde{S}_i(z) \ \dd z \right)\ \cap \ \textup{ker} \int \no{\tilde{\gamma}_{1,m-2}(z) \ee^{\frac{-1}{\kk+\dcox} \tsroot{m-1}}(z)} \ \dd z\ \\
     &\cap \bigcap_{i=m}^{n} \textup{ker} \int \ee^{\frac{-1}{\kk+\dcox} \tilde{\alpha}_{i}}(z) \ \dd z\ \cap\  \textup{ker} \int \ee^{\frac{1}{2}c + \frac{1}{2}d}(z) \ dz \subset \widetilde{\uaffvoa{\kk+\dcox}{\cartan}} \otimes \widetilde{\Pi} \otimes \bigotimes_{\alpha \in (\proots)_0 \setminus \hroot_0} \widetilde{\bgvoa_\alpha}. \notag
\end{align}

The first $m-1$ screening operators act nontrivially only on $\widetilde{\uaffvoa{\kk+\dcox}{\cartan}} \otimes\bigotimes_{j=1}^{m-2}\bigotimes_{i=1}^{j} \widetilde{\bgvoa_{\sroot{i,j}}}$ and describe precisely the Wakimoto realisation of $\uQHR{\kk}{\slnpone}{\bar{f}^{(m-1)}}$ as per \eqref{eq:WakScreenersHook} using the isomorphism between $\widetilde{\uaffvoa{\kk+\dcox}{\cartan}} \otimes \bigotimes_{j=1}^{m-2}\bigotimes_{i=1}^{j} \widetilde{\bgvoa_{\sroot{i,j}}}$ and $\uaffvoa{\kk+\dcox}{\cartan} \otimes \bigotimes_{j=1}^{m-2}\bigotimes_{i=1}^{j} \bgvoa_{\sroot{i,j}}$ defined by removing the tildes on all fields as in \eqref{eq:hookffisos}. Thus the fields in $\widetilde{\uaffvoa{\kk+\dcox}{\cartan}} \otimes \bigotimes_{j=1}^{m-2}\bigotimes_{i=1}^{j} \widetilde{\bgvoa_{\sroot{i,j}}}$ that appear in the Wakimoto realisation of $\uQHR{\kk}{\slnpone}{\fhook{m}}$ can be identified with fields in $\uQHR{\kk}{\slnpone}{\bar{f}^{(m-1)}} \cong \uQHR{\kk}{\slnpone}{\fhook{m-1}}$. 

\begin{theorem} \label{thm:HookToHook}
Let $\kk$ be generic. There exists an embedding 
\begin{equation}
    \uQHR{\kk}{\slnpone}{\fhook{m}} \hookrightarrow \uQHR{\kk}{\slnpone}{\fhook{m-1}} \otimes \Pi \otimes \bigotimes_{i=2}^{m-1} \bgvoa_{\sroot{i,m-1}}
\end{equation}
whose image is specified by
\begin{equation}
   \uQHR{\kk}{\slnpone}{\fhook{m}} \cong \textup{\normalfont ker} \int  \ee^{A_m}(z) \ \dd z,
\end{equation}
where
\begin{align}
     A_m(z) = &\ \frac{1}{2}\left(1-(\kk+\dcox)\inner{\fwt{m-1}}{\fwt{m-1}} \right) c(z) + \frac{1}{2}d(z)  - \fwt{m-1}(z) \\
    &+ \frac{1}{2} \sum_{\substack{\alpha,\alpha' \in (\proots)_0 \setminus \hroot_0 \\ \alpha+\alpha'=\hroot_0}} \no{\beta_\alpha (z) \beta_{\alpha'}(z)\ee^{-c}(z)} 
    +\frac{1}{2} \sum_{\substack{\alpha,\alpha',\alpha'' \in (\proots)_0 \setminus \hroot_0 \\ -\alpha+\alpha'+\alpha''=\hroot_0}} \no{\gamma_\alpha(z) \beta_{\alpha'} (z) \beta_{\alpha''}(z)\ee^{-c}(z)}. \notag
\end{align}
Here, the fields $\fwt{m-1}(z)$, $\beta_\alpha(z)$ and $\gamma_\alpha(z)$, for $\alpha =\sroot{i,j}$ with $ j \neq m-1$, act on $\uQHR{\kk}{\slnpone}{\fhook{m-1}}$ via the Wakimoto realisation of $\uQHR{\kk}{\slnpone}{\fhook{m-1}}$.
\end{theorem}
\begin{proof}
This statement is proved in an identical fashion to \cref{thm:mintoaff}.
\end{proof}

\noindent This confirms a conjecture in \cite{FehThesis}. The minimal-to-affine case corresponds to choosing $m=n+1$, and the $m=2$ case is included in the above theorem by omitting $\bigotimes_{i=2}^{m-1} \bgvoa_{\sroot{i,m-1}}$ and observing that $(\proots)_0 \setminus \hroot_0$ is empty.

\subsection{Nongeneric levels} \label{subsec:Nongeneric}

In this section, we prove the existence of the hook-to-hook inverse reduction embedding for all noncritical $\kk$, rather than just nongeneric $\kk$ as in \cref{thm:HookToHook}. The key observation is that the proofs of the previous section apply equally well when considering the vertex algebras involved as vertex algebras over the ring $U = \CC[k]$ where $k$ is a formal parameter. Then, subject to observing that the resulting vertex algebra embedding can be specialised to any given $k=\kk$, we get our desired inverse reduction.

To begin, let $\uQHR{U}{\slnpone}{\fhook{m}}$ be the vertex operator algebra obtained by the $U$-version of quantum hamiltonian reduction, see \cite[Sec.~2.2]{GenScreen20}. Let $\uaffvoa{U+\dcox}{\cartan}$ be the vertex operator algebra obtained from $\uaffvoa{\kk+\dcox}{\cartan}$ by replacing $\kk$ with the formal parameter $k$. Then, by \cite[Cor.~4.12]{GenScreen20}, there is an embedding
\begin{equation}
    \uQHR{U}{\slnpone}{\fhook{m}} \hookrightarrow \uaffvoa{U+\dcox}{\cartan} \otimes \bigotimes_{\alpha \in (\proots)_0} \bgvoa_\alpha
\end{equation}
whose image is specified by
\begin{equation}
    \uQHR{U}{\slnpone}{\fhook{m}} \cong \left( \bigcap_{i=1}^{m-1} \textup{ker} \int S^U_i(z) \ \dd z \right) \cap \left(  \bigcap_{i=m}^{n} \textup{ker} \int \ee^{\frac{-1}{k+\dcox} \alpha_{i}}(z) \ \dd z \right),
\end{equation}
where $S^U_i(z)$ is given by the same formula as $S_i(z)$, except with $\kk$ replaced with $k$, and treated as an operator on $\uaffvoa{U+\dcox}{\cartan} \otimes \bigotimes_{\alpha \in \proots} \bgvoa_\alpha$. An important consequence of this is that the image of $\uQHR{U}{\slnpone}{\fhook{m}}$ in $\uaffvoa{U+\dcox}{\cartan} \otimes \bigotimes_{\alpha \in (\proots)_0} \bgvoa_\alpha$ contains no poles in $k$.

By the same argument as in \cref{subsec:HooktoHook}, except with $\kk$ replaced with $k$ throughout, we obtain an embedding of vertex algebras over $U$
\begin{equation}
    \psi^U: \uQHR{U}{\slnpone}{\fhook{m}} \hookrightarrow \uQHR{U}{\slnpone}{\fhook{m-1}} \otimes \Pi \otimes \bigotimes_{i=2}^{m-1} \bgvoa_{\sroot{i,m-1}}.
\end{equation}

We would like to specialise $\psi^U$ to any noncritical level $\kk$. To do so, let $\set{W_i(z)}$ be a PBW basis of $\uQHR{U}{\slnpone} {\fhook{m}}$ (such a basis exists by \cite[Thm.~4.1]{KacQua04}) and consider the image $\set{\psi^U(W_i(z))}$ of each field. Define 
\begin{equation}
    \psi^\kk: \uQHR{\kk}{\slnpone}{\fhook{m}} \rightarrow \uQHR{\kk}{\slnpone}{\fhook{m-1}} \otimes \Pi \otimes \bigotimes_{i=2}^{m-1} \bgvoa_{\sroot{i,m-1}}.
\end{equation}
to be the vertex algebra morphism determined by, treating $\set{W_i(z)}$ as a PBW basis of $\uQHR{\kk}{\slnpone}{\fhook{m}}$,
\begin{equation}
    \psi^\kk: W_i(z) \mapsto \psi^U(W_i(z)) \Big|_{k=\kk}.
\end{equation}
This is a vertex algebra morphism, that is it preserves the operator product expansions, because $\psi^U$ is. What remains is to show that $\psi^\kk$ is injective. 

Firstly, we can choose $\set{W_i(z)}$ such that the fields $\psi^U(W_i(z))$ have at least one non-$k$ dependent summand: Consider the Miura map $\mu^U: \uQHR{U}{\slnpone}{\fhook{m}} \hookrightarrow \uaffvoa{U}{(\slnpone)_0}$ of \cite{KacQua03}. By \cite[Thm.~4.1]{KacQua04}, choose $\set{W_i(z)}$ such that $W_i(z) = J^{u_i}(z) + \cdots$ where $J^{u_i}(z) \in \uaffvoa{U}{(\slnpone)_0}$ is $k$-independent. Composing the Miura map $\mu^U$ with the Wakimoto realisation $\rho^U$ of $\uaffvoa{U}{(\slnpone)_0}$ and the bosonisation of $\bgvoa_{1,m}$ defines an embedding
\begin{equation}
    \hat{\psi}^U: \uQHR{U}{\slnpone}{\fhook{m}} \hookrightarrow \uaffvoa{U+\dcox}{\cartan} \otimes \Pi \otimes \bigotimes_{\alpha \in (\proots)_0 \setminus \alpha_{1,m}} \bgvoa_{\alpha}.
\end{equation}

Since the Wakimoto realisation of $J^{u_i}(z)$ has a $k$-independent term \cite{Fre05Opers}, so too does the image of $W_i(z)$ under $\hat{\psi}^U$. The composition of $\rho^U \circ \mu^U$ is equal to the Wakimoto realisation of $\uQHR{U}{\slnpone}{\fhook{m}}$ (see \cite[Sec.~4.5]{GenScreen20}). Recall that $\psi^U$ is obtained by taking the Wakimoto realisation of $\uQHR{U}{\slnpone}{\fhook{m}}$, bosonising $\bgvoa_{1,m}$ and rewriting the fields in terms of the tilded fields. All of the tilded fields contain $k$-independent leading terms.

Thus, the images of the fields $W_i(z)$ in $\widetilde{\uaffvoa{U+\dcox}{\cartan}}  \otimes \widetilde{\Pi} \otimes \bigotimes_{\alpha \in (\proots)_0 \setminus \alpha_{1,m}} \widetilde{\bgvoa_{\alpha}}$, and therefore the images with respect to $\psi^U$, have $k$-independent leading terms. So the specialisation $\psi^{\kk}$ of $\psi^U$ to any noncritical $\kk$ maps $W_i(z)$ to something nonzero.

Finally, $\psi^\kk$ is injective because the composition of it with the Wakimoto realisation of $\uQHR{\kk}{\slnpone}{\fhook{m-1}}$ is equal to the composition of the specialisation $\hat{\psi}^\kk$ with rewriting the fields in terms of tilded fields and the isomorphisms \eqref{eq:hookffisos}, which is injective. 

\begin{theorem} \label{thm:HookToHookNogen}
Let $\kk$ be noncritical and $m \in \ZZ$ satisfying $2 \leq m \leq n+1$. There exists an embedding 
\begin{equation} \label{eq:HooksFromHooks}
    \uQHR{\kk}{\slnpone}{\fhook{m}} \hookrightarrow \uQHR{\kk}{\slnpone}{\fhook{m-1}} \otimes \Pi \otimes \bigotimes_{i=2}^{m-1} \bgvoa_{\sroot{i,m-1}}.
\end{equation}
\end{theorem}

\begin{corollary}
Let $\kk$ be noncritical and $m \in \ZZ_\geq 1$ satisfying $1 \leq m \leq n+1$. There exists an embedding 
\begin{equation} \label{eq:AffFromHooks}
    \uaffvoa{\kk}{\slnpone} \hookrightarrow \uQHR{\kk}{\slnpone}{\fhook{m}}  \otimes \Pi^{\otimes (n+1-m)} \otimes  \bgvoa^{\otimes \frac{1}{2}(n+m-2)(n+1-m)}.
\end{equation}
\end{corollary}
\begin{proof}
    Compose the inverse reduction embeddings of \cref{thm:HookToHookNogen}.
\end{proof}

\section{Outlook}  \label{sec:Outlook}
While we have proved the existence (and under certain circumstances provided a construction; see \cref{app:ExpExp}) of inverse reduction embeddings amongst hook-type $\slnpone$ W-algebras, there are plenty of non-hook-type $\slnpone$ W-algebras. For example, the rectangular $\slfour$ W-algebra $\uQHR{\kk}{\slfour}{\frec}$, where $\frec = f_{1,2}+f_{3,4}$, corresponding to the partition $(2,2)$ whose operator product expansions are detailed in \cite[Sec.~2.2]{Ada21Certain}. Indeed screening operators corresponding to a Wakimoto realisation of $\uQHR{\kk}{\slfour}{\frec}$ are presented in \cite[App.~A.1]{CreCorr20}. Using these to identify inverse reductions involving $\uQHR{\kk}{\slfour}{\frec}$ is the topic of ongoing work.

There are also non-type-A W-algebras that might be accessible using the free-field approach to inverse quantum hamiltonian reductions used here. One known example is the case of the subregular $\mathfrak{sp}_4$ W-algebra. A realisation of it in terms of the regular $\mathfrak{sp}_4$ W-algebra and the half lattice was proposed in \cite{Beem21}. An identical screening operator argument to that used in \cite[Thm.~3.3]{Feh21c} and \cref{thm:mintoaff} provides an alternative construction of this realisation. 

The representation-theoretic content of the general hook-type inverse reductions remains to be explored. Such embeddings have been used previously to study relaxed and logarithmic W-algebra modules in, for example, \cite{AdaRea17,AdaRea20,Feh21c,ACG21}. We expect that the new inverse reductions \eqref{eq:HooksFromHooks} can produce similar insights into the representation theory of hook-type W-algebras, in particular $\uaffvoa{\kk}{\slnpone}$. 

When the embeddings \eqref{eq:HooksFromHooks} descend to an embedding of simple quotients is a difficult but important question. In the known $\slnpone$ cases, the inverse reduction descends to simple quotients for almost all $\kk$. There however, these embeddings can be written down and used directly to study singular vectors of a specific form. It is possible that the question of when \eqref{eq:HooksFromHooks} descends to simple quotients can be addressed without knowing explicit formulae, but this is to be considered in future work.

\appendix
\section{Explicit expressions} \label{app:ExpExp}

One useful feature of the argument presented in this paper is that if the Wakimoto realisations of the involved W-algebras are known explicitly (corresponding to the particular choice of screening operators in \cref{subsec:HooktoHook}), then the inverse reduction can be written down explicitly. We conclude with an illustrative example: the inverse reduction
\begin{equation}
    \uaffvoa{\kk}{\slfour} \hookrightarrow \uQHR{\kk}{\slfour}{\fmin}\otimes \Pi \otimes \bigotimes_{i=2}^{3} \bgvoa_{\sroot{i,3}} .
\end{equation}

The minimal W-algebra $\uQHR{\kk}{\slfour}{\fmin}$ is strongly-generated by $9$ fields denoted by $L(z), H(z), E(z), F(z), J(z), P^{1,\pm}(z)$ and $P^{2,\pm}(z)$. The first of these fields, $L(z)$, is a conformal field whose associated central charge is, using \eqref{eq:hookcc},
\begin{equation}
    \cc^{(2,1,1)}=-\frac{3\kk (2\kk+3)}{\kk+4}.
\end{equation}
With respect to $L(z)$, each of $H(z),E(z),F(z),J(z),P^{1,\pm}(z)$ and $P^{2,\pm}(z)$ are primary fields and their conformal dimensions are $1,1,1,1,\frac{3}{2}$ and $\frac{3}{2}$ respectively. $H(z),E(z),F(z)$ generate a subalgebra of $\uQHR{\kk}{\slfour}{\fmin}$ isomorphic to $\uaffvoa{\kk+1}{\sltwo}$ and $J(z)$ generates a Heisenberg subalgebra of weight $\kk+2$. The singular operator product expansions of the generating fields, excluding those implied by the preceding discussion, are \cite{Ara17Orb}
\begin{equation} \label{ope:sl4min}
\begin{gathered}
    H(z)P^{i,\pm}(w) \sim \frac{(3-2i)P^{i,\pm}(w)}{z-w}, \quad
    J(z)P^{i,\pm}(w) \sim \frac{\pm P^{i,\pm}(w)}{(z-w)}, \\
    E(z)P^{2,-}(w) \sim \frac{P^{1,-}(w)}{z-w}, \quad 
    E(z)P^{2,+}(z) \sim \frac{-P^{1,+}(w)}{z-w}, \quad 
    F(z)P^{1,-}(w) \sim \frac{P^{2,-}(w)}{z-w}, \quad 
    F(z)P^{1,+}(w) \sim \frac{-P^{2,+}(w)}{z-w}, \\
    P^{1,-}(z)P^{1,+}(w) \sim \frac{-2 (\kk+2) E(w)}{(z-w)^2} + \frac{2\no{J(w)E(w)} - (\kk+2)\partial E(w)}{z-w}, \\
    P^{2,-}(z)P^{2,+}(w) \sim \frac{-2 (\kk+2) F(w)}{(z-w)^2} + \frac{2\no{J(w)F(w)} - (\kk+2)\partial F(w)}{z-w}, \\
    P^{1,-}(z)P^{2,+}(w) \sim \frac{-2(\kk+1)(\kk+2)\wun(w)}{(z-w)^3} + \frac{2(\kk+1)J(w) - (\kk+2)H(w)}{(z-w)^2} \hspace{18em} \\
        \hspace{3em} +\frac{(\kk+4)L(w) - 2\no{E(w)F(w)} - \frac{1}{2}\no{H(w)H(w)} + \no{H(w)J(w)} - \frac{3}{2}\no{J(w)J(w)} - \frac{\kk}{2} \partial H(w) + (\kk+1) \partial J(w)}{z-w}, \\
    P^{1,+}(z)P^{2,-}(w) \sim \frac{2(\kk+1)(\kk+2)\wun(w)}{(z-w)^3} + \frac{2(\kk+1)J(w) + (\kk+2)H(w)}{(z-w)^2} \hspace{18em} \\
        \hspace{3em} +\frac{-(\kk+4)L(w) + 2\no{E(w)F(w)} + \frac{1}{2}\no{H(w)H(w)} + \no{H(w)J(w)} + \frac{3}{2}\no{J(w)J(w)} + \frac{\kk}{2} \partial H(w) +(\kk+1) \partial J(w)}{z-w}.
\end{gathered}
\end{equation}

The Wakimoto realisation of $\uaffvoa{\kk}{\slfour}$ is the vertex operator algebra embedding $\uaffvoa{\kk}{\slfour} \hookrightarrow \uaffvoa{\kk+4}{\mathfrak{h}}\otimes \bigotimes_{\alpha \in \proots} \bgvoa_\alpha$ determined by (see \cite[Thm.~5.1]{Fre05Opers})
\begin{equation}
\begin{aligned}
    e_1(z) \mapsto& \beta_1(z) + \no{\beta_{1,2}(z) \gamma_2(z)} + \no{\beta_{1,3}(z) \gamma_{2,3}(z)}, \quad 
    e_2(z) \mapsto \beta_2(z) + \no{\beta_{2,3}(z) \gamma_3(z)}, \quad
    e_3(z) \mapsto \beta_3(z),\\
    h_1(z) \mapsto& \alpha_1(z) + 2 \no{\beta_1(z) \gamma_1(z)}- \no{\beta_2(z) \gamma_2(z)}  + \no{\beta_{1,2} \gamma_{1,2}(z)}- \no{\beta_{2,3}(z) \gamma_{2,3}(z)} +     \no{\beta_{1,3}(z) \gamma_{1,3}(z)} , \\
    h_2(z) \mapsto& \alpha_2(z) - \no{\beta_1(z) \gamma_1(z)}+  2 \no{\beta_2(z) \gamma_2(z)} - \no{\beta_3(z) \gamma_3(z)} + \no{\beta_{1,2}(z) \gamma_{1,2}(z)} +       \no{\beta_{2,3}(z) \gamma_{2,3}(z)},\\
    h_3(z) \mapsto& \alpha_3(z)- \no{\beta_2(z) \gamma_2(z)} + 2 \no{\beta_3(z) \gamma_3(z)} - \no{\beta_{1,2}(z) \gamma_{1,2}(z)}  + \no{\beta_{2,3}(z)   \gamma_{2,3}(z)} + \no{\beta_{1,3}(z) \gamma_{1,3}(z)},\\
    f_1(z) \mapsto& -\no{\alpha_1(z) \gamma_1(z)} - \no{\beta_1(z)\gamma_1(z) \gamma_1(z)} + \no{\beta_2(z) \gamma_{1,2}(z)} + \no{\beta_{2,3}(z) \gamma_{1,3}(z)} - (\kk+2) \partial\gamma_1(z),\\ 
    f_2(z) \mapsto&  - \no{\alpha_2(z) \gamma_2(z)} - \no{\beta_2(z)\gamma_2(z) \gamma_2(z)}+ \no{\beta_1(z) \gamma_1(z) \gamma_2(z)}  \\
        &\hspace{3em} - \no{\beta_{1,2}(z) \gamma_2(z) \gamma_{1,2}(z)}- \no{\beta_1(z)\gamma_{1,2}(z)}  + \no{\beta_3(z) \gamma_{2,3}(z)} - (\kk+1) \partial \gamma_2(z),\\
    f_3(z) \mapsto& -\no{\alpha_3(z) \gamma_3(z)} - \no{\beta_3(z)\gamma_3(z) \gamma_3(z)}+ \no{\beta_2(z) \gamma_2(z) \gamma_3(z)} \\
        &\hspace{3em} +\no{\beta_{1,2}(z) \gamma_3(z) \gamma_{1,2}(z)}- \no{\beta_{2,3}(z) \gamma_3(z) \gamma_{2,3}(z)}-\no{\beta_{1,3}(z) \gamma_3(z) \gamma_{1,3}(z)}  \\
        &\hspace{5em} - \no{\beta_2(z) \gamma_{2,3}(z)}- \no{\beta_{1,2}(z) \gamma_{1,3}(z)}  - \kk\ \partial \gamma_3(z).
\end{aligned}
\end{equation}
One can also check directly that the fields above are in the kernel of the screening operators $\int S_i(z) \dd z$ for $i=1,2,3$ (see \eqref{eq:WakScreenersExp}). Bosonising $\bgvoa_\theta = \bgvoa_{1,3}$ defines an embedding $\uaffvoa{\kk}{\slfour} \hookrightarrow \uaffvoa{\kk+4}{\mathfrak{h}}\otimes \Pi \otimes \bigotimes_{\alpha \in \proots \setminus \theta} \bgvoa_\alpha$. Rewriting the fields in terms of
the tilded fields using the definition of $\tilde{\alpha}_i(z)$, \eqref{eq:tildeHalfLat} and \eqref{eq:tildeGhosts}, we obtain

\begin{align*}
    e_1(z) \mapsto&\ \no{\tilde{\gamma}_{2,3}(z) \ee^{\tilde{c}}(z)}, \quad 
    e_2(z) \mapsto \tilde{\beta}_2(z) + \no{\tilde{\beta}_{2,3}(z) \tilde{\gamma}_3(z)} \quad
    e_3(z) \mapsto \tilde{\beta}_3(z),\\
    h_1(z) \mapsto&\ \frac{3}{4}\tilde{\alpha}_1(z) - \frac{1}{2}\tilde{\alpha}_2(z) - \frac{3}{4}\tilde{\alpha}_3(z) + \frac{1}{2}\tilde{d}(z) - \frac{1}{8}(3\kk+16)\tilde{c}(z) \\
        &\hspace{3em} + 2 \no{\tilde{\beta}_{1}(z)\tilde{\gamma}_{1}(z)} + \no{\tilde{\beta}_{1,2}(z)\tilde{\gamma}_{1,2}(z)} - \no{\tilde{\beta}_{2}(z)\tilde{\gamma}_{2}(z)} - \no{\tilde{\beta}_{2,3}(z)\tilde{\gamma}_{2,3}(z)}, \\
    h_2(z) \mapsto&\ \tilde{\alpha}_2(z) - \no{\tilde{\beta}_{1}(z)\tilde{\gamma}_{1}(z)} + \no{\tilde{\beta}_{1,2}(z)\tilde{\gamma}_{1,2}(z)} + 2 \no{\tilde{\beta}_{2}(z)\tilde{\gamma}_{2}(z)} + \no{\tilde{\beta}_{2,3}(z)\tilde{\gamma}_{2,3}(z)} - \no{\tilde{\beta}_{3}(z)\tilde{\gamma}_{3}(z)},\\
    h_3(z) \mapsto&\ -\frac{1}{4}\tilde{\alpha}_1(z) - \frac{1}{2}\tilde{\alpha}_2(z) + \frac{1}{4}\tilde{\alpha}_3(z) + \frac{1}{2}\tilde{d}(z) + \frac{1}{8}(5\kk+16)\tilde{c}(z) \\
        &\hspace{3em} - \no{\tilde{\beta}_{1,2}(z)\tilde{\gamma}_{1,2}(z)} - \no{\tilde{\beta}_{2}(z)\tilde{\gamma}_{2}(z)} + \no{\tilde{\beta}_{2,3}(z)\tilde{\gamma}_{2,3}(z)}+2\no{\tilde{\beta}_{3}(z)\tilde{\gamma}_{3}(z)},
\end{align*}
\newpage
\begin{align*}
    f_1(z) \mapsto&\ -\no{\tilde{\beta}_{1}(z)\tilde{\gamma}_{1}(z)\tilde{\gamma}_{1}(z)} + 2\no{\tilde{\beta}_{1}(z)\tilde{\beta}_{2,3}(z)\tilde{\gamma}_{1}(z) \ee^{-\tilde{c}}(z)}+\no{\tilde{\beta}_{1,2}(z)\tilde{\beta}_{2,3}(z)\tilde{\gamma}_{1,2}(z) \ee^{-\tilde{c}}(z)} + \no{\tilde{\beta}_2(z) \tilde{\gamma}_{1,2}(z)} \\
        &\hspace{3em} - \no{\tilde{\beta}_2(z) \tilde{\beta}_3(z) \ee^{-\tilde{c}}(z)} - \no{\tilde{\beta}_2(z) \tilde{\beta}_{2,3}(z) \tilde{\gamma}_2(z) \ee^{-\tilde{c}}(z)}  - \frac{1}{8}(11\kk+16)\no{\tilde{\beta}_{2,3}(z) \tilde{c}(z) \ee^{-\tilde{c}}(z)}\\
        &\hspace{5em}+\frac{1}{2}\no{\tilde{\beta}_{2,3}(z) \tilde{d}(z) \ee^{-\tilde{c}}(z)}+\frac{3}{4}\no{\tilde{\beta}_{2,3}(z) \tilde{\alpha}_1(z) \ee^{-\tilde{c}}(z)} - \frac{1}{2}\no{\tilde{\beta}_{2,3}(z) \tilde{\alpha}_2(z) \ee^{-\tilde{c}}(z)} \\
        &\hspace{6em} - \frac{3}{4} \no{\tilde{\beta}_{2,3}(z) \tilde{\alpha}_3(z) \ee^{-\tilde{c}}(z)} - \no{\tilde{\gamma}_1(z) \tilde{\alpha}_1(z)} + (\kk+1) \no{\partial \tilde{\beta}_{2,3}(z) \ee^{-\tilde{c}}(z)} - (\kk+2) \partial \tilde{\gamma}_1(z),\\ 
    f_2(z) \mapsto&\ -\no{\tilde{\beta}_{1}(z)\tilde{\gamma}_{1,2}(z)} + \no{\tilde{\beta}_1(z) \tilde{\gamma}_1(z) \tilde{\gamma}_2(z)} - \no{\tilde{\beta}_2(z) \tilde{\gamma}_2(z) \tilde{\gamma}_2(z)}\\
        &\hspace{3em} + \no{\tilde{\beta}_{3}(z)\tilde{\gamma}_{2,3}(z)} - \no{\tilde{\gamma}_2(z) \tilde{\alpha}_2(z)} -\no{\tilde{\beta}_{1,2}(z) \tilde{\gamma}_2(z) \tilde{\gamma}_{1,2}(z)}-(\kk+1) \partial \tilde{\gamma}_2(z),\\
    f_3(z) \mapsto&\ \no{\tilde{\beta}_1(z) \tilde{\beta}_2(z) \ee^{-\tilde{c}}(z)} - \no{\tilde{\beta}_{1,2}(z) \tilde{\beta}_{1,2}(z) \tilde{\gamma}_{1,2}(z) \ee^{-\tilde{c}}(z)} + \no{\tilde{\beta}_{1,2}(z) \tilde{\alpha}_3(z) \ee^{-\tilde{c}}(z)} - \no{\tilde{\beta}_2(z) \tilde{\gamma}_{2,3}(z)} \\
        &\hspace{3em} + \no{\tilde{\beta}_{2}(z) \tilde{\gamma}_{2}(z) \tilde{\gamma}_{3}(z)} - \no{\tilde{\beta}_{3}(z) \tilde{\gamma}_{3}(z) \tilde{\gamma}_{3}(z)} + \frac{1}{4} \no{\tilde{\gamma}_3(z) \tilde{\alpha}_1(z)} + \frac{1}{2} \no{\tilde{\gamma}_3(z) \tilde{\alpha}_2(z)} \\
        &\hspace{5em}  - \frac{1}{4} \no{\tilde{\gamma}_3(z) \tilde{\alpha}_3(z)}- \frac{1}{8}(5\kk+16) \no{\tilde{\gamma}_3(z) \tilde{c}(z)} - \frac{1}{2}\no{\tilde{\gamma}_3(z) \tilde{d}(z)} + \no{\tilde{\beta}_{1,2}(z)\tilde{\gamma}_3(z)\tilde{\gamma}_{1,2}(z)} \\
        &\hspace{6em} - \no{\tilde{\beta}_{2,3}(z)\tilde{\gamma}_3(z)\tilde{\gamma}_{2,3}(z)} + (\kk+2) \no{\partial \tilde{\beta}_{1,2}(z) \ee^{-\tilde{c}}(z)} + \kk \partial \tilde{\gamma}_3(z).
\end{align*}

We know that the fields of $\widetilde{\uaffvoa{\kk+4}{\mathfrak{h}}} \otimes \widetilde{\bgvoa_{1,1}} \otimes \widetilde{\bgvoa_{1,2}} \otimes \widetilde{\bgvoa_{2,2}}$ appearing above can be identified with fields in $\uQHR{\kk}{\slfour}{\fmin}$ via its Wakimoto realisation (and the isomorphisms \eqref{eq:hookffisos}). For example, by grouping the terms in the image of $f_3(z)$, we obtain
\begin{align} \label{eq:sl4mintoaffF3}
    f_3(z) \mapsto&\ \no{\left(\tilde{\beta}_1(z) \tilde{\beta}_2(z) - \tilde{\beta}_{1,2}(z) \tilde{\beta}_{1,2}(z) \tilde{\gamma}_{1,2}(z)  + \tilde{\beta}_{1,2}(z) \tilde{\alpha}_3(z) + (\kk+2)\partial \tilde{\beta}_{1,2}(z)\right) \ee^{-\tilde{c}}(z)} - \no{\left(\tilde{\beta}_2(z)\right) \tilde{\gamma}_{2,3}(z)} \\
    &\hspace{2em} +  \no{\left(\tilde{\beta}_{2}(z)\tilde{\gamma}_{2}(z)+\tilde{\beta}_{1,2}(z)\tilde{\gamma}_{1,2}(z) + \frac{1}{4}\tilde{\alpha}_1(z)+ \frac{1}{2}  \tilde{\alpha}_2(z) - \frac{1}{4} \tilde{\alpha}_3(z) \right) \tilde{\gamma}_{3}(z)} \notag\\
    &\hspace{3em}- \no{\tilde{\beta}_{3}(z) \tilde{\gamma}_{3}(z) \tilde{\gamma}_{3}(z)}  - \frac{1}{8}(5\kk+16) \no{\tilde{\gamma}_3(z) \tilde{c}(z)} - \frac{1}{2}\no{\tilde{\gamma}_3(z) \tilde{d}(z)}  - \no{\tilde{\beta}_{2,3}(z)\tilde{\gamma}_3(z)\tilde{\gamma}_{2,3}(z)} + \kk \partial \tilde{\gamma}_3(z). \notag
\end{align}

Meanwhile, the images of $E(z), P^{1,+(z)}, J(z)$ and $H(z)$ in $\uaffvoa{\kk+4}{\mathfrak{h}} \otimes \bgvoa_{1,1} \otimes \bgvoa_{1,2} \otimes \bgvoa_{2,2}$ under the Wakimoto realisation of $\uQHR{\kk}{\slfour}{\fmin}$ defined by the screening operators \eqref{eq:WakAffandMin} are given by
\begin{equation}
\begin{aligned}
    E(z) &\mapsto \beta_2(z), \\
    J(z) &\mapsto \frac{1}{2}\alpha_1(z) - \frac{1}{2}\alpha_3(z) + \no{\beta_1(z)\gamma_1(z)} + \no{\beta_{1,2}(z)\gamma_{1,2}(z)}, \\
    H(z) &\mapsto \alpha_2(z) - \no{\beta_1(z)\gamma_1(z)} + \no{\beta_{1,2}(z)\gamma_{1,2}(z)} + 2 \no{\beta_2(z)\gamma_2(z)}, \\
    P^{1,+}(z) &\mapsto -\beta_1(z) \beta_2(z) + \beta_{1,2}(z) \beta_{1,2}(z) \gamma_{1,2}(z) - \beta_{1,2}(z) \alpha_3(z) - (\kk+2)\partial \beta_{1,2}(z).
\end{aligned}
\end{equation}
Applying the isomorphisms $\bgvoa_{1,1} \cong \widetilde{\bgvoa_{1,1}}, \bgvoa_{1,2} \cong \widetilde{\bgvoa_{1,2}}, \bgvoa_{2,2} \cong \widetilde{\bgvoa_{2}}$ and $\uaffvoa{\kk+4}{\mathfrak{h}} \cong \widetilde{\uaffvoa{\kk+4}{\mathfrak{h}}}$ defined by adding tildes to the fields, we can therefore identify
\begin{equation} \label{eq:tildedWaksl4Min}
\begin{aligned}
    E(z) &\leftrightarrow \tilde{\beta}_2(z), \\
    J(z) &\leftrightarrow \frac{1}{2}\tilde{\alpha}_1(z) - \frac{1}{2}\tilde{\alpha}_3(z) + \no{\tilde{\beta}_1(z)\tilde{\gamma}_1(z)} + \no{\tilde{\beta}_{1,2}(z)\tilde{\gamma}_{1,2}(z)}, \\
    H(z) &\leftrightarrow \tilde{\alpha}_2(z) - \no{\tilde{\beta}_1(z)\tilde{\gamma}_1(z)} + \no{\tilde{\beta}_{1,2}(z)\tilde{\gamma}_{1,2}(z)} + 2 \no{\tilde{\beta}_2(z)\tilde{\gamma}_2(z)}, \\
    P^{1,+}(z) &\leftrightarrow -\tilde{\beta}_1(z) \tilde{\beta}_2(z) + \tilde{\beta}_{1,2}(z) \tilde{\beta}_{1,2}(z) \tilde{\gamma}_{1,2}(z) - \tilde{\beta}_{1,2}(z) \alpha_3(z) - (\kk+2)\partial \tilde{\beta}_{1,2}(z).
\end{aligned}
\end{equation}
The fields from $\widetilde{\uaffvoa{\kk+4}{\mathfrak{h}}} \otimes \widetilde{\bgvoa_{1,1}} \otimes \widetilde{\bgvoa_{1,2}} \otimes \widetilde{\bgvoa_{2,2}}$ appearing in the image of $f_3(z)$ in \eqref{eq:sl4mintoaffF3} can then be written in terms of the $\uQHR{\kk}{\slfour}{\fmin}$ fields using \eqref{eq:tildedWaksl4Min}. We therefore have
\begin{align}
    f_3(z) \mapsto&\ -\no{P^{1,+}(z) \ee^{-\tilde{c}}(z)} - \no{E(z)\tilde{\gamma}_{2,3}(z)} + \frac{1}{2}\no{(J(z)+H(z)) \tilde{\gamma}_3(z)}  \\
        &\hspace{2em} - \no{\tilde{\beta}_{3}(z)\tilde{\gamma}_3(z)\tilde{\gamma}_{3}(z)} - \frac{1}{8}(5\kk+16) \no{\tilde{\gamma}_3(z)\tilde{c}(z)} - \frac{1}{2} \no{\tilde{\gamma}_3(z)\tilde{d}(z)} - \no{\tilde{\beta}_{2,3}(z)\tilde{\gamma}_3(z)\tilde{\gamma}_{2,3}(z)} + \kk \partial\tilde{\gamma}_3(z). \notag
\end{align}

Making similar identifications for the rest of the  $\uaffvoa{\kk}{\slfour}$ fields, we obtain our inverse reduction embedding
\begin{equation} \label{eq:sl4MintoAff}
\begin{aligned}
    e_1(z) \mapsto&\ \no{\gamma_{2,3}(z) \ee^{c}(z)}, \quad 
    e_2(z) \mapsto E(z) + \no{\beta_{2,3}(z)\gamma_3(z)}, \quad
    e_3(z) \mapsto \beta_3(z),\\
    h_1(z) \mapsto&\ \frac{3}{2}J(z) - \frac{1}{2} H(z) + \frac{1}{2}d(z) - \frac{1}{8}(3\kk+16)c(z)- \no{\beta_{2,3}(z)\gamma_{2,3}(z)},\\
    h_2(z) \mapsto&\ H(z) + \no{\beta_{2,3}(z)\gamma_{2,3}(z)} - \no{\beta_{3}(z)\gamma_{3}(z)},\\
    h_3(z) \mapsto&\ -\frac{1}{2} J(z) - \frac{1}{2} H(z) + \frac{1}{2}d(z) + \frac{1}{8}(5\kk+16) c(z) + \no{\beta_{2,3}(z)\gamma_{2,3}(z)} + 2 \no{\beta_{3}(z)\gamma_{3}(z)},\\
    f_1(z) \mapsto&\ P^{1,-}(z) - \no{E(z)\beta_3(z) \ee^{-c}(z)} + \frac{1}{2}\no{\left(3 J(z) - H(z)\right) \beta_{2,3}(z) \ee^{-c}(z)} \\
        &\hspace{3em} - \frac{1}{8} (11 \kk+16) \no{\beta_{2,3}(z) c(z) \ee^{-c}(z)} + \frac{1}{2} \no{\beta_{2,3}(z) d(z) \ee^{-c}(z)} + (\kk + 1) \no{\partial \beta_{2,3}(z) \ee^{-c}(z)},\\ 
    f_2(z) \mapsto&\ F(z) + \no{\beta_3(z) \gamma_{2,3}(z)},\\
    f_3(z) \mapsto&\ -\no{P^{1,+}(z) \ee^{-c}(z)} - \no{E(z)\gamma_{2,3}(z)} + \frac{1}{2}\no{(J(z)+H(z)) \gamma_3(z)}  \\
        &\hspace{3em} - \no{\beta_{3}(z)\gamma_3(z)\gamma_{3}(z)} - \frac{1}{8}(5\kk+16) \no{\gamma_3(z)c(z)} - \frac{1}{2} \no{\gamma_3(z)d(z)} - \no{\beta_{2,3}(z)\gamma_3(z)\gamma_{2,3}(z)} + \kk \partial \gamma_3(z),
\end{aligned}
\end{equation}
where we have used the isomorphisms $\widetilde{\Pi} \cong \Pi$, $\widetilde{\bgvoa_{2,3}} \cong \bgvoa_{2,3}$ and $\widetilde{\bgvoa_{3,3}} \cong \bgvoa_{3,3}$ to remove the remaining tildes. The image of the remaining strong generating fields of $\uaffvoa{\kk}{\slfour}$ can be obtained from the above using \eqref{eq:affOPE}. Observe too that all expressions for embeddings in this section contain no poles in $\kk$, and the image of strong generating fields always contains a $\kk$-independent leading term.

% \subsection{A.1 title} \label{subsec:A.1} 
% \newpage
% \subsection{A.2 title} \label{subsec:A.2} 

% \newpage
% \raggedright
%\bibliography{subreg}
%\bibliographystyle{unsrt}

\raggedright
% \bibliographystyle{unsrt}
% \bibliography{bib.bib}
\providecommand{\opp}[2]{\textsf{arXiv:\mbox{#2}/#1}}\providecommand{\pp}[2]{\textsf{arXiv:#1
  [\mbox{#2}]}}

\end{document}